\documentclass[11pt]{amsart}

\usepackage{amsmath,amssymb,txfonts,amsthm}

\newtheorem{thm}{Theorem}[section]

\newtheorem{prop}[thm]{Proposition}
\newtheorem{lem}[thm]{Lemma}
\newtheorem{cor}[thm]{Corollary}

\theoremstyle{definition}
\newtheorem{exam}[thm]{Example}

\newtheorem{defi}[thm]{Definition}
\newtheorem{rem}[thm]{Remark}

\begin{document}

\author{Morimichi Kawasaki}

\address[Morimichi Kawasaki]{Center for Geometry and Physics, Institute for Basic Science (IBS), Pohang 37673, Republic of Korea (previous), Research Institute for Mathematical Sciences, Kyoto University\\
Kyoto 606-8502, Japan (present)}
\email{kawasaki@kurims.kyoto-u.ac.jp}
\title{Function theoretical applications of Lagrangian spectral invariants}
\maketitle
\begin{abstract}
Entov and Polterovich considered the concept of heaviness and superheaviness by the Oh-Schwarz spectral invariants.
The Oh-Schwarz spectral invariants are defined in terms of the Hamiltonian Floer theory.
In this paper, we define heaviness and superheaviness by spectral invariants defined in terms of the Lagrangian Floer theory and provide their applications.
As one of them, we define a relative symplectic capacity which measures the existence of Hamiltonian chord between two disjoint Lagrangian submanifolds and provide an upper bound of it in a special case. 
We also provide applications to non-degeneracy of spectral norms, a relative energy capacity inequality, fragmentation norm and non-displaceability.
\end{abstract}

\large

\section{Principal Results}

\subsection{Introduction}

In a series of papers by Entov and Polterovich \textit{et al.}, they constructed a fertile theory by asymptotic Oh-Schwarz spectral invariants.
Polterovich and Rosen's book ``Function theory on symplectic manifolds'' \cite{PR} is a good survey of their works.

In their theory, (abstractly) displaceable subsets play important roles in many situations.
In the present paper, we consider a relative version of their theory \textit{i.e.}
we replace abstract displaceability in their theory with relative displaceability from a fixed Lagrangian submanifold.
Our main tool is a Lagrangian spectral invariant whose properties are studied by Leclercq and Zapolsky \cite{LZ}.

We provide applications of our observation to non-displaceability, a relative version of energy capacity inequality, Poisson bracket invariant for open covers and existence problem of a Hamiltonian chord between two disjoint submanifolds.


\subsection{Notions on Hamiltonian isotopies and monotonicity of Lagrangian submanifolds}
Let $(M,\omega)$ be a symplectic manifold.
For a Hamiltonian function $H\colon M\to\mathbb{R}$ with compact support, we define the \textit{Hamiltonian vector field} $X_H$ associated with $H$ by
\[\omega(X_H,V)=-dH(V)\text{ for any }V \in \mathcal{X}(M),\]
where $\mathcal{X}(M)$ is the set of smooth vector fields on $M$.

Let $S^1$ denote $\mathbb{R}/\mathbb{Z}$.
For a (time-dependent) Hamiltonian function $H\colon S^1\times M\to\mathbb{R}$ with compact support and for $t \in S^1$, we define $H_t\colon M\to\mathbb{R}$ by $H_t(x)=H(t,x)$. 
Let $X_H^t$ denote the Hamiltonian vector field associated with $H_t$ and $\{\phi_H^t\}_{t\in\mathbb{R}}$ denote the isotopy generated by $X_H^t$ such that $\phi^0=\mathrm{id}$.
Let $\phi_H$ denote $\phi_H^1$ and $\phi_H$ is called \textit{the Hamiltonian diffeomorphism generated by } $H$.
For $x\in M$, let $\gamma_H^x{\colon}[0,1]\to{M}$ denote the path defined by $\gamma_H^x(t)=\phi_H^t(x)$.

Let $X$, $Y$ be subsets of a symplectic manifold $(M,\omega)$.
$X$ is \textit{displaceable} from $Y$ if there exists a Hamiltonian function $H\colon S^1\times M\to\mathbb{R}$ such that $\phi_H(X)\cap\bar{Y}=\emptyset$, where $\bar{Y}$ is the topological closure of $Y$.
$X$ is \textit{non-displaceable} from $Y$ otherwise.
$X$ is \textit{(abstractly) displaceable} if $X$ is displaceable from $X$ itself.

For Hamiltonian functions $F,G\colon S^1\times M\to\mathbb{R}$, let $F{\natural}G\colon S^1\times M\to\mathbb{R}$ denote a Hamiltonian function defined by $(F{\natural}G)(t,x)=F(t,x)+G(t,(\phi_F^t)^{-1}(x))$.
For a Hamiltonian function $H\colon S^1\times M\to\mathbb{R}$, let $\bar{H}\colon S^1\times M\to\mathbb{R}$ denote a Hamiltonian function defined by $\bar{H}(t,x)=-H(t,(\phi_H^t)^{-1}(x))$.
Note that $\phi_{F\natural G}^t=\phi_F^t\phi_G^t$ and $\phi_{\bar{H}}^t=(\phi_H^t)^{-1}$.

A Hamiltonian function $H\colon S^1\times M\to\mathbb{R}$ is \textit{normalized} if $\int_MH_t\omega^m=0$ for any $t$, where $m=\frac{1}{2}\mathrm{dim}(M)$.

For a Hamiltonian function $H\colon S^1\times M\to \mathbb{R}$ with compact support on a symplectic manifold $M$, we define \textit{the Hofer length} $||H||$ of $H$ by
\[||H||=\int_0^1(\max_{x\in M}H_t(x)-\min_{x\in M}H_t(x))dt.\]
For subsets $X,Y$ of $M$, we define \textit{the displacement energy} $E(X;Y)$ of $X$ from $Y$ by
\[E(X;Y)=\inf\{ ||H|| ; H\in C_c^\infty(S^1\times M), \phi_H^1(X)\cap\bar{Y}=\emptyset \} ,\]
where $\bar{Y}$ is the topological closure of $Y$.
If $X$ is non-displaceable from $Y$, we define $E(X;Y)=+\infty$.
We also define another invariant $\bar{E}(X;Y)$ by
\[\bar{E}(X;Y)=\min\{E(X;X),E(X;Y)\}.\]

For a closed Lagrangian submanifold $L$, we have the following two homomorphisms.
\[\omega\colon \pi_2(M,L)\to\mathbb{R} \text{, the \textit{symplectctic area}},\]
\[\mu\colon \pi_2(M,L)\to\mathbb{Z} \text{, the \textit{Maslov index}}.\]

For a positive number $\lambda$, $L$ is said to be $\lambda$-\textit{monotone} if
\[\omega(A)=\lambda\cdot\mu(A)\text{ for any }A\in\pi_2(M,L).\]

$L$ is said to be \textit{monotone} if $L$ is $\lambda$-\textit{monotone} for some positive number $\lambda$.
Note that $L$ is $\lambda$-\textit{monotone} for any positive number $\lambda$ if $L$ is a Lagrangian submanifold with $\pi_2(M,L)=0$.

The \textit{minimal Maslov number} $N_L$ of $L$ is defined to be the positive generator of the subgroup $\mu(\pi_2(M,L))$ of $\mathbb{Z}$ if it is nontrivial, and we set $N_L=\infty$ otherwise.

\subsection{Lagrangian heaviness and superheaviness}

Let $L$ be a monotone Lagrangian submanifold of a closed symplectic manifold $(M,\omega)$ with $N_L\geq 2$ and $QH_\ast(L)\neq0$.

For an element $\mathtt{a}$ of the quantum homology $QH_\ast(L)$ and a continuous function $H\colon S^1\times M\to\mathbb{R}$,
Leclercq and Zapolsky defined the Lagrangian spectral invariant $c^L(\mathtt{a},H)\in\mathbb{R}$ (See Section \ref{prel section}).

For an idempotent $\mathtt{a}$ of the quantum homology $QH_\ast(L)$, we define the functional  $\zeta_\mathtt{a}^L{\colon}$ $C^0(S^1\times M)\to\mathbb{R}$ as the stabilization of $c^L(\mathtt{a},\cdot);$
\[\zeta_\mathtt{a}^L(H)=\lim_{k\to+\infty}\frac{c^L(\mathtt{a},H^{\natural k})}{k},\]
where a function $H^{\natural k}\colon S^1\times M\to\mathbb{R}$ is defined by
\[H^{\natural k}=\underbrace{H\natural\cdots\natural H}_{k}.\]

\begin{defi}\label{definition of heavy}
Let $L$ be a monotone Lagrangian submanifold of a closed symplectic manifold $(M,\omega)$ with $N_L\geq 2$, $QH_\ast(L)\neq0$ and $\mathtt{a}$ an idempotent of the quantum homology $QH_\ast(L)$.
A closed subset $X$ of $(M,\omega)$ is said to be $(L,\mathtt{a})$-\textit{heavy} if
\[\zeta_\mathtt{a}^L(H)\geq\inf_{S^1\times X}H\text{ for any continuous function } H\colon S^1\times M\to\mathbb{R}.\]

$X$ is said to be $(L,\mathtt{a})$-\textit{superheavy} if
\[\zeta_\mathtt{a}^L(H)\leq\sup_{S^1\times X}H\text{ for any continuous function } H\colon S^1\times M\to\mathbb{R}.\]
A closed subset $X$ of $(M,\omega)$ is called $L$-\textit{(super)heavy} in the Lagrangian sense if $X$ is $(L,\mathtt{a})$-(super)heavy for some idempotent $\mathtt{a}$ of $QH_\ast (L)$.
\end{defi}
By an argument similar to \cite{EP09}, we see that $X$ is $(L,\mathtt{a})$-heavy if $X$ is $(L,\mathtt{a})$-superheavy.

By Theorem 39 of \cite{LZ}, we can prove the following theorem.
\begin{thm}\label{product}
Let $L_1,L_2$ be monotone Lagrangian submanifolds with same monotonicity constants of closed symplectic manifolds $(M_1,\omega_1)$, $(M_2,\omega_2)$ with $N_{L_1}\geq 2, N_{L_2}\geq2$, respectively.
Let $X_1, X_2$ be a $(L_1,\mathtt{a}_1)$-heavy, $(L_2,\mathtt{a}_2)$-heavy subset, respectively.
Then $X_1\times X_2$ is a $(L_1\times L_2,\mathtt{a}_1\times \mathtt{a}_2)$-heavy subset.
\end{thm}

\subsection{Examples}

We provide some examples of heavy and superheavy subsets.
A priority of Entov-Polterovich's theory is that we can prove non-displaceability of singular subsets.
Entov and Polterovich proved that a stem which can be singular is superheavy (\cite{EP09}).
We pose a generalization of concept of stem.

\begin{defi}\label{L stem}
Let $L$ be a subset of a closed symplectic manifold $M$.
Let $\mathbb{A}$ be a finite-dimensional Poisson-commutative subspace of $C^\infty(M)$ and $\Phi\colon M\to\mathbb{A}^\ast$ be the moment map defined by $\langle \Phi(x),F \rangle =F(x)$. 
 A non-empty fiber $\Phi^{-1}(p)$, $p \in \mathbb{A}^\ast$ is called \textit{an $L$-stem} of $\mathbb{A}$ if each non-empty fibers $\Phi^{-1}(q)$ with $q \neq p$ is displaceable from $L$ or $\Phi^{-1}(q)$ itself.
 If a subset $X$ of $M$ is an $L$-stem of a finite-dimensional Poisson-commutative subspace of $C^\infty(M)$, $X$ is called just \textit{an $L$-stem}.
\end{defi}

For example, $L$ itself is an $L$-stem if $L$ is a closed subset of $M$.
Any stem in the original sense of \cite{EP09} is also an $L$-stem for any $L$.

\begin{thm}\label{stem shv}
Let $L$ be a monotone Lagrangian submanifold of a closed symplectic manifold $(M,\omega)$ with $N_L\geq2$ and $QH_\ast(L)\neq0$.
Then any $L$-stem is $(L,\mathtt{a})$-superheavy for any non-trivial idempotent $\mathtt{a}$.
\end{thm}

It is a natural problem whether (super)heaviness with respect to some idempotent implies (super)heaviness with respect to another idempotent.
We provide the following two results.
In Section \ref{module section}, we provide more generalized theorems.

\begin{cor}\label{module find 1}
Let $L$ be a monotone Lagrangian submanifold of a closed symplectic manifold $(M,\omega)$ with $N_L\geq2$ and $QH_\ast(L)\neq0$.
Let $\mathtt{e}^\prime$ be an idempotent of $QH_\ast(L)$.
Then 
\begin{itemize}
\item
Any $[M]$-superheavy subset $X$ is $(L,\mathtt{e}^{\prime})$-superheavy,
\item
Any $(L,\mathtt{e}^{\prime})$-heavy subset $Y$ is $[M]$-heavy, in particular, non-displaceable from $Y$ itself.
\end{itemize}
\end{cor}

\begin{cor}\label{module find 2}
Let $L$ be a monotone Lagrangian submanifold of a closed symplectic manifold $(M,\omega)$ with $N_L\geq2$ and $QH_\ast(L)\neq0$.
Let $\mathtt{e}^\prime$ be a non-trivial idempotent of $QH_\ast(L)$.
Then any $(L,\mathtt{e}^{\prime})$-heavy subset $X$ is $(L,[L])$-heavy and any $(L,[L])$-superheavy subset $Y$ is $(L,\mathtt{e}^{\prime})$-superheavy.
\end{cor}

The proof of Corollary \ref{module find 2} is similar to Theorem 1.5 of \cite{EP09} and thus we omit the proof.

We provide many examples of heavy or superheavy subsets when $M$ is a 2-torus.
For a positive number $R$, we define the 2-torus $T_R^2$ by $T_R^2=\mathbb{R}/R\mathbb{Z}\times\mathbb{R}/\mathbb{Z}$ with the coordinates $(x,y)$.
For $s\in\mathbb{R}/R\mathbb{Z}$ and $t\in\mathbb{R}/\mathbb{Z}$ , let $L_{\mathrm{m}}^s$ and $L_{\mathrm{l}}^t$ be a meridian curve $\{x=s\}$ and a longitude curve $\{y=t\}$ , respectively.

We have many examples of heavy, superheavy subsets of $T_R^2$.

\begin{prop}\label{torus summary}
For any positive number $R$, the following propositions hold on $T_R^2$.
\begin{enumerate}
\item
$L_{\mathrm{m}}^{s^\prime}$ is $(L_{\mathrm{m}}^s,[L_{\mathrm{m}}^s])$-superheavy if $s=s^\prime$.

\item
$L_{\mathrm{m}}^{s^\prime}$ is not $(L_{\mathrm{m}}^s,[L_{\mathrm{m}}^s])$-heavy if $s\neq s^\prime$.

\item
$L_{\mathrm{l}}^t$ is $(L_{\mathrm{m}}^s,[L_{\mathrm{m}}^s])$-heavy for any $s,t$.

\item
$L_{\mathrm{l}}^t$ is not $(L_{\mathrm{m}}^s,[L_{\mathrm{m}}^s])$-superheavy for any $s,t$.

\item
$L_{\mathrm{m}}^{s^\prime}\cup L_{\mathrm{l}}^t$ is $(L_{\mathrm{m}}^s,[L_{\mathrm{m}}^s])$-superheavy for any $s,s^\prime,t$.
\end{enumerate}
\end{prop}

Later, we provide a generalization of (3) of Proposition \ref{torus summary} (See Proposition \ref{Riemannian}).


\subsection{Applications to non-displaceability}

As an application of Lagrangian spectral invariants to non-displaceablity, we have the following theorem.

\begin{thm}\label{hv from shv}
Let $L$ be a monotone Lagrangian submanifold of a closed symplectic manifold $(M,\omega)$ with $N_L\geq 2$.
Let $X$, $Y$ be a $(L,\mathtt{a})$-heavy subset, a $(L,\mathtt{a})$-superheavy subset of $M$, respectively.
Then $X$ is non-displaceable from $Y$.
\end{thm}

As a corollary of Theorem \ref{product}, (3) of Proposition \ref{torus summary}, Theorem \ref{stem shv} and Theorem \ref{hv from shv}, we obtain the following corollary.
\begin{cor}\label{product with torus}
Let $L$ be a monotone Lagrangian submanifold of a closed symplectic manifold $(M,\omega)$ with $N_L\geq2$ and $QH_\ast(L)\neq0$.
Let $X$, $Y$ be $L$-stems.
Then the subset $X\times L_{\mathrm{l}}^t$ of $M\times T_R^2$ is non-displaceable from $Y\times L_{\mathrm{m}}^s$ for any $s,t$.
\end{cor}

\begin{rem}
Since both of $X\times L_{\mathrm{m}}^s$ and $Y\times L_{\mathrm{l}}^t$ are displaceable by symplectomorphisms which are isotopic to the identity, they are not superheavy in the original sense of Entov and Polterovich.
Thus we cannot apply (iii) of Theorem 1.4 of \cite{EP09} to Corollary \ref{product with torus}.
\end{rem}

Entov and Polterovich proved that any moment map has at least one fiber which is non-displaceable from itself (Theorem 2.1 of \cite{EP06}, Theorem 6.1.8 of \cite{PR}).
We improve their result under an assumption that $M$ admits a Lagrangian submanifold satisfying some conditions.

\begin{cor}\label{at least one fiber}
Let $L$ be a monotone Lagrangian submanifold of a closed symplectic manifold $(M,\omega)$ with $N_L\geq2$ and $QH_\ast(L)\neq0$.
For any moment map $\Phi\colon M\to\mathbb{R}^k$,
there exists a point $y_0$ of $\Phi(L)$ such that
$\Phi^{-1}(y_0)$ is non-displaceable from $L$ and $\Phi^{-1}(y_0)$ itself.
\end{cor}

\begin{proof}
Assume on the contrary that $\Phi^{-1}(y)$ is displaceable from $L$ or $\Phi^{-1}(y)$ itself for any point $y$ of $\Phi(L)$.
Note that, for any point $y$ of $\Phi(M)\setminus\Phi(L)$, $\Phi^{-1}(y)$ is disjoint from $L$, in particular displaceable from $L$ .
Hence, any fiber of $\Phi$ is an $L$-stem and thus, by Theorem \ref{stem shv} and Theorem \ref{hv from shv}, non-displaceable from $L$.
This is a contradiction.
\end{proof}

\begin{rem}
By an argument similar to Subsection 9.2 of \cite{PR}, we can regard Corollary \ref{at least one fiber} as a corollary of Theorem \ref{positive poisson}.
\end{rem}





\subsection{Non-degeneracy of spectral norms and the energy capacity inequality}

In the process of proving (3) of Proposition \ref{torus summary}, we prove the non-degeneracy of spectral norms.

\begin{thm}\label{non-deg of spec norm}
Let $L$ be a monotone Lagrangian submanifold of a closed symplectic manifold $(M,\omega)$ with $N_L\geq2$ and $QH_\ast(L)\neq0$.
Then for any idempotent $\mathtt{a}$ of the quantum homology $QH_\ast(L)$ and  any Hamiltonian function $H\colon S^1 \times M\to\mathbb{R}$ with compact support such that $\phi_H(L)\neq L$,
\[c^L(\mathtt{a},H)+c^L(\mathtt{a},\bar{H})> 0.\]
\end{thm}

On the case of the Oh-Schwarz spectral invariants,
the non-degeneracy of spectral invariants is proved by Schwarz \cite{Sch} when $(M,\omega)$ is symplectically ashperical and $H$ is non-degenerate.
After that, Oh \cite{Oh05} generalized it to the case in which $(M,\omega)$ is a general closed symplectic manifold and $H$ is a general Hamiltonian function. 
Frauenfelder and Schlenk \cite{FS} proved it when $(M,\omega)$ is a weakly convex symplectic manifold and $H$ is a general Hamiltonian function. 

Polterovich and Rosen gave another proof of Oh's result using a Poisson bracket inequality
(Proposition 4.6.2 of \cite{PR}).
Similarly to Polterovich and Rosen, we use Poisson bracket inequality to prove Theorem \ref{non-deg of spec norm}.
However, our Poisson bracket inequality is weaker than their one and we have to improve their argument.

\begin{rem}
The author tried to prove the non-degeneracy of spectral norms by an idea similar to Schwarz or Oh or Frauenfelder-Schlenk's one, but it failed because of some difficulties.
Here, we explain one of them.
In Schwarz and Oh's proof, it is important to prove the existence of a (broken) Floer trajectory between two orbits which represent $[M]$ and $[\mathrm{pt}]$, where $[\mathrm{pt}]$ is the homology class representing a point 
(See Subsection 2.3 of \cite{Sch} and Theorem 5.4 of \cite{Oh05}).
In our case, $QH_\ast(L)\equiv H_\ast(L)$ does not hold in general, and thus we cannot take a homology class in $QH_\ast(L)$ representing $[\mathrm{pt}]$ in general.
\end{rem}

As its application, we obtain a relative version of the energy capacity inequality which is announced by Lisi and Rieser (Theorem 4.1 of \cite{LR}) and proved by Humili$\grave{\mathrm{e}}$re, Leclercq and Seyfaddini
for cotangent bundles (Lemma 7 of \cite{HLS}).

Let $(N,\omega)$ be a symplectic manifold and $L$ a closed subset of $N$.

\begin{defi}
An autonomous Hamiltonian function $H\colon N\to\mathbb{R}$ with compact support is \textit{$L$-simple} if

\begin{description}
\item[(1)]
There exists a compact subset $K$ of $N$ such that
\[K\cap L\neq\emptyset\text{ and }H|_{N\setminus K}\equiv 0,\]

\item[(2)]
There exists an open subset $W$ of $N$ and a constant $m(H)$ such that
\[W\cap L\neq\emptyset\text{ and }H|_W\equiv m(H),\]

\item[(3)]
The only critical values of $H$ are $0$ and $m(H)$.

\end{description}
\end{defi}

A Hamiltonian function $H\colon N\to\mathbb{R}$ is \textit{$L$-slow} if $H$ is $L$-simple and all of its Hamiltonian chord of length at most $1$ from $L$ to $L$ are constant
\textit{i.e.}
for any $0<t\leq1$ and any smooth path $z\colon[0,1]\to N$ with
$\dot{z}(t)=X_H^t(z(t))$ and $z(0),z(1)\in L$,
$z(s)=z(0)$ holds  for any $0<s\leq t$.
Then, we define the Lisi-Rieser-type capacity $C_{LR}(N;L)$ by
\[C_{LR}(N;L)=\sup\{m(H); H\colon N\to\mathbb{R} \text{ is $L$-slow}\}.\]


\begin{thm}\label{energy capacity ineq}
Let $L$ be a Lagrangian submanifold of a closed symplectic manifold $(M,\omega)$ with $\omega(\pi_2(M,L))=0$, $N_L\geq2$ and $QH_\ast(L)\neq0$
(Note that $\omega(\pi_2(M,L))=0$ implies monotonicity of $L$).
Then, for an open subset $U$ of $M$ with $L\cap U\neq\emptyset$,
\[C_{LR}(U;L\cap U)\leq \bar{E}(U;L).\]

\end{thm}

\begin{rem}
Our definition of simplicity follows \cite{FGS} and it is slightly different from Lisi and Rieser's original one.
In their original definition, ``The only critical values of $H$ are $0$ and $m(H)$'' is replaced by ``$0\leq H(x)\leq m(H)$ for all $x\in M$''.
Note that Theorem \ref{energy capacity ineq} does not imply Lisi and Rieser's original energy capacity inequality.
\end{rem}

\subsection{Application to Poisson bracket invariant}

Let $Q_N$ denote the cube $[0,1]^N$.
For a partition of unity $\vec{F}=\{F_1,\ldots,F_N\}$ on a symplectic manifold $(M,\omega)$, we define the magnitude $\kappa_{cl}(\vec{F})$ of its Poisson noncommutativity by
\[\kappa_{cl}(\vec{F})=\max_{x,y\in Q_N}||\{\sum_ix_iF_i,\sum_jy_jF_j\}||.\]
For an open cover $\mathcal{U}$ of $M$, we define the Poisson bracket invariant $\mathrm{pb}(\mathcal{U})=\inf \kappa_{cl}(\vec{F})$.
Here the infimum is taken over all partition $\vec{F}$ of the unity subordinated to $\mathcal{U}$.

For an open cover $\mathcal{U}=\{U_1,\ldots,U_N\}$ of $M$, we define the displacement energy $\bar{E}(\mathcal{U};L)$ of $\mathcal{U}$ by $\bar{E}(\mathcal{U};L)=\max_i\bar{E}(U_i;L)$.

The following theorem is a relative version of Theorem 3.1 of \cite{P12} and Theorem 9.2.2 of \cite{PR} (see also Theorem 1.8 of \cite{EPZ}).

\begin{thm}\label{positive poisson}
Let $L$ be a monotone Lagrangian submanifold of a $2m$-dimensional closed symplectic manifold $(M,\omega)$ with $N_L\geq2$ and $QH_\ast(L)\neq0$.
Let $\mathcal{U}=\{U_1,\ldots,U_N\}$ be an open cover such that each $U_i$ is displaceable from $L$ or $U_i$ itself.
Then
\[\mathrm{pb}(\mathcal{U})\cdot \bar{E}(\mathcal{U};L)\geq (2N^2)^{-1}.\]
In particular, $\bar{E}(\mathcal{U};L)>0$ and
\[\mathrm{pb}(\mathcal{U})\geq (2N^2\cdot\bar{E}(\mathcal{U};L))^{-1}>0.\]
\end{thm}
Polterovich used the Poisson bracket inequality efficiently to prove the original (abstract) version of Theorem \ref{positive poisson}, but our Poisson bracket inequality (Proposition \ref{pb ineq}) is too complicated to prove Theorem \ref{positive poisson}.
For this difficulty, we avoid to use a partial symplectic quasi-state.
Instead of this, we use a Lagrangian spectral invariant directly.
The reason why Polterovich used a partial quasi-state is that it has the semi-homogeneity property.
Instead of the semi-homogeneity property of a partial symplectic quasi-state, we use the Hamiltonian shift property of a Lagrangian spectral invariant.

We provide an application of Theorem \ref{positive poisson}.
For an open cover $\mathcal{U}=\{U_i\}_i$ of a manifold $M$, let $\mathcal{U}\times\mathcal{U}$ denote the open cover $\{U_i\times U_j\}_{i,j}$ of $M\times M$

\begin{exam}\label{torus cov}
Let $\{U_1,\ldots,U_M,V_1,\ldots, V_N\}$, $\{W_1,W_2\}$ be open cover of a circle $S^1=\mathbb{R}/\mathbb{Z}$ such that
\begin{itemize}
\item
$U_i\neq S^1$, $V_j\neq  S^1$ and $W_k\neq  S^1$ for any $i$, $j$ and $k$,
\item
$0\in U_i$ for any $i=1.\ldots,M$ and $0\notin V_j$ for any $j=1,\ldots,N$.
\end{itemize}
We define a cover $\mathcal{U}$ of $T_1^2=S^1\times S^1$ by 
\[\mathcal{U}=\{U_i\times W_k,V_j\times S^1\}_{i,j,k}.\]
Then, any element of $\mathcal{U}$ is displaceable from $\{0\}\times S^1$.
Thus, by Theorem \ref{positive poisson}, $\mathrm{pb}(\mathcal{U})$ is positive.
Similarly, we can prove that $\mathrm{pb}(\mathcal{U}\times\mathcal{U})$ is also positive.
On the other hand, we can easily confirm that $\mathrm{pb}(\mathcal{U}^\prime)=0$ where 
\[\mathcal{U^\prime}=\{U_i\times S^1,V_j\times S^1\}_{i,j}.\]
\end{exam}

\begin{rem}
Kaoru Ono pointed out that we can prove positivity of $\mathrm{pb}(\mathcal{U})>0$ by the original Polterovich's theorem  when $\{U_1,\ldots,U_M\}$ is not an open cover of a circle $S^1=\mathbb{R}/\mathbb{Z}$.
\end{rem}

We note that the assumption $QH_\ast(L)\neq0$ in Theorem \ref{positive poisson} is essential.
$QH_\ast(L)\neq0$ implies non-displaceability of $L$.
We provide the following proposition if $L$ is displaceable.

\begin{prop}\label{disp pb inv}
Let $X$ be a closed subset of a closed symplectic manifold $(M,\omega)$.
Assume that $X$ is displaceable from $X$ itself.
Then, there exists an open cover $\mathcal{U}=\{U_1,\ldots,U_N\}$ of $M$ such that any $U_i$ is displaceable from $X$ and $\mathrm{pb}(\mathcal{U})=0$.

\end{prop}

\subsection{Hamiltonian chords between two disjoint subsets}



For subsets $Y_0, Y_1$ of an open symplectic manifold $(N,\omega)$.
For a Hamiltonian function $H\colon S^1\times N\to\mathbb{R}$ with compact support and a homotopy class $\alpha\in\pi_1(M,Y_0\cup Y_1)$,
We define a set $\mathsf{Ch}(H;Y_0,Y_1,\alpha)$ of Hamiltonian chords of $H$ from $Y_0$ to $Y_1$ in $\alpha$ by
\[\mathsf{Ch}(H;Y_0,Y_1,\alpha)=\{ z\colon 
[0,1]\to N ;  \dot{z}(t)=X_H^t(z(t)), z(0)\in Y_0 \text{ and } z(1)\in Y_1, [z]=\alpha\}.\]

For subsets $Y_0,Y_1$ and $Z$ of an open symplectic manifold  $(N,\omega)$ and a homotopy class $\alpha$ of $\pi_1(N,Y_0\cup Y_1)$,
we define the relative symplectic capacity $C_{BEP}(N,Y_0,Y_1,Z,\alpha)$ by
\[C_{BEP}(N,Y_0,Y_1,Z,\alpha)=\inf \{K>0 ; \forall H\in \mathcal{H}_K(N,Z), \mathsf{Ch}(H;Y_0,Y_1,\alpha)\neq \emptyset \},\]
where 
\[\mathcal{H}_K(N,Z)=\{ H\in C_c^\infty(S^1\times N) ; \inf_{S^1\times Z} H \geq K \}.\]

\begin{rem}
Buhovsky, Entov and Polterovich studied existence problem of Hamiltonian chords between two disjoint subsets under some robust restriction on the $C^0$-profile of the Hamiltonian function (\cite{BuEP}).
They made a method to find a Hamiltonian chord whose length is equal to or smaller than 1.
On the other hand, an upper bound of $C_{BEP}$ provides  a Hamiltonian chord whose length is equal to 1.
\end{rem}

For $R=(R_1,\ldots,R_n)\in(\mathbb{R}_{>0})^n$, let $\mathbb{A}(R)$ denote an annulus $(0,R_1)\times(0,R_2)\times(0,R_n)\times(\mathbb{R}/\mathbb{Z})\times\cdots\times(\mathbb{R}/\mathbb{Z})$ with the coordinates $(p_1,\ldots,p_n,q_1,\ldots,q_n)$.
We fix a symplectic form $dp_1\wedge dq_1+\cdots+dp_n\wedge dq_n$ on $\mathbb{A}(R)$.
For $s=(s_1,\ldots,s_n)\in(\mathbb{R}/\mathbb{Z})^n$ for any $i$, define $\mathbb{A}(R)_s$ by
\[\mathbb{A}(R)_s=\{(p,q)\in\mathbb{A}(R); q=s\}.\]

For $R=(R_1,\ldots,R_n)\in(\mathbb{R}_{>0})^n$, $s=(s_1,\ldots,s_n)\in\mathbb{R}^n$ and subsets $Y,Z$ of $M$, we define a relative symplectic capacity $C(M,Y,Z,R,s)$ by
\[C(M,Y,Z,R,s)=C_{BEP}(M\times \mathbb{A}(R),Y\times \mathbb{A}(R)_0,Y\times \mathbb{A}(R)_s,Z\times \mathbb{A}(0),\alpha_s).\]
Here $\alpha_s$ is the homotopy class of $\pi_1(M\times \mathbb{A}(R),Y\times \mathbb{A}(R)_0\cup Y\times \mathbb{A}(R)_s)$ represented by the path $\gamma_s\colon [0,1]\to M\times\mathbb{A}(R)$ defined by $t\mapsto(p_0,([s_1t],\ldots,[s_nt]),(0,\ldots,0))$ where $p_0$ is the fixed base point of $Y$.

\begin{thm}\label{main monotone just}
Let $L$ be a monotone Lagrangian submanifold of a $2m$-dimensional closed symplectic manifold $(M,\omega)$ with $N_L\geq2$ and $QH_\ast(L)\neq0$.
Let $Z$ be an $(L,\mathtt{a})$-heavy subset of $M$ for some idempotent $\mathtt{a}$.
Then for any $R=(R_1,\ldots,R_n)\in(\mathbb{R}_{>0})^n$ and any $s=(s_1,\ldots,s_n)\in\mathbb{R}^n$,
\[C(M,L,Z,R,s)\leq2\sum_{i=1}^nR_i\cdot|s_i|+\lambda(m+n).\]
\end{thm}

In \cite{Ka}, the author used the fact ``the meridian curve is heavy in the sense of Entov and Polterovich" (\cite{EP09}).
We use (3) of Proposition \ref{torus summary} instead of that fact.

In contrast, we prove the following proposition when $Z$ is displaceable from $L$ (recall that $(L,\mathtt{a})$-heavy subset is non-displaceable from $L$).
\begin{prop}\label{displaceable main}
Let $Z$ be a displaceable compact of a closed symplectic manifold $(M,\omega)$.
Assume $R_k\cdot s_k>E(Z;L)$ for some $k$, where $E(Z;L)$ is the displacement energy.
Then
\[C(M,L,Z,R,s)=+\infty.\]
\end{prop}

The proof is also a Lagrangian analogue of \cite{Ka}.

\subsection*{Technical remark}

In the case of Oh-Schwarz spectral invariants, (abstractly) displaceable open subsets have some special property.
The source of these properties is the bounded spectrum condition for (abstractly) displaceable subsets (Proposition \ref{abstract disp bounded spec ham 2}),

In our Lagrangian spectral invariants, we prove the bounded spectrum condition for relatively displaceable subsets (Proposition \ref{Ostrover}) and it enables us to make relative versions of some known results.

However, it is not sufficient to replace the bounded spectrum condition for (abstractly) displaceable subsets with the one for relatively displaceable subsets.
There is a large technical difference between Oh-Schwarz spectral invariants and Lagrangian spectral invariants.
Oh-Schwarz spectral invariants are invariant under the action of Hamiltonian diffeomorphisms, but the same property does not hold for Lagrangian ones.
This difference makes some technical difficulties.

In the author's opinion, the easiest way to understand this difficulty is comparing our Poisson bracket inequality (Proposition \ref{pb ineq}) with the original one (Theorem 3.1 of \cite{P12} and Proposition 4.6.1 of \cite{PR}).
For more precise explanation, see Remark \ref{technical remark on pb}.

\subsection*{Overview of the present paper}


In Section \ref{prel section}, we prepare some notions and define Lagrangian spectral invariants.
In Section \ref{module section}, we prove Corollary \ref{module find 1}.

In Section \ref{Ostrover section}, we define the bounded spectrum condition for open subsets of a symplectic manifold.
We prove abstractly or relatively displaceable open subsets satisfy the bounded spectrum condition (Propositions \ref{Ostrover}, \ref{abstract disp bounded spec ham 1}).
In Section \ref{pb section}, we prove Theorem \ref{positive poisson} and Proposition \ref{disp pb inv}.

In Section \ref{Properties of partial quasi-state}, we provide some properties of asymptotic Lagrangian spectral invariants and prove Theorem \ref{hv from shv}..
In Section \ref{PB ineq section}, we formulate a Lagrangian version of Poisson bracket inequality (Proposition \ref{pb ineq}) which is necessary for proving some results in this paper.

In Section \ref{stem section}, we prove some properties of asymptotic Lagrangian spectral invariants and Theorem \ref{stem shv}.
In Section \ref{non-deg section}, we prove Theorem \ref{non-deg of spec norm}.
In Section \ref{Proof of the energy capacity inequality}, we prove a relative version of the energy capacity inequality (Theorem \ref{energy capacity ineq}) by Theorem \ref{non-deg of spec norm}.

In this section, we also prove Proposition \ref{torus summary}.
To prove heaviness of a longitude curve with respect to a meridian curve ((3) of Proposition \ref{torus summary}), we use Theorem \ref{non-deg of spec norm}.

In Section \ref{prf of main theorem}, we prove Theorem \ref{main monotone just} by (3) of Theorem \ref{torus summary}.
In Section \ref{flex section}, we prove Proposition \ref{displaceable main}.

\subsection*{Acknowledgment}
For the author, this paper is the first work where he uses the Lagrangian Floer theory.
The author thanks Professor Yong-Geun Oh and Professor Kaoru Ono for replying some elementary questions on the Lagrangian Floer theory.
He also thanks Yoosik Kim, Weonmo Lee, Ryuma Orita, Professor Leonid Polterovich, Fumihiko Sanda and Frol Zapolsky for faithful discussions and comments.
Especialy, Ryuma Orita checked this paper carefully and pointed out a lot of careless mistakes.

He is supported by IBS-R003-D1 and JP18J00765.

\section{Preliminary}\label{prel section}


In the present paper, we follow the notion of \cite{LZ} and we consider only the case in which the coefficient on Floer theory is $\mathbb{Z}/2\mathbb{Z}$.

For a subset $X$ of $M$, we define
\[\Omega_0(X)=\{z\colon [0,1]\to M, \text{ smooth map }; z(0),z(1)\in X, [z]=0_X\},\]
 where $0_X$ is the trivial element of $\pi_1(M,X)$.

To consider the Lagrangian Floer theory, we consider the following covering space $\tilde{\Omega}_0(L)$ of $\Omega_0(L)$.
Note that our explanation is a rough sketch and omit the definitions of some words.
See \cite{LZ} and \cite{Z} for them and details.

Let $z\in\Omega_0(L)$.
Two cappings $\hat{z},\hat{z}^\prime\colon D^2\setminus\{1\}\to M$ of $z$ are \textit{equivalent} if the preglued map $\hat{z}\natural(-\hat{z}^\prime)$ defines the trivial element in $\pi_2(M,L)$, where $-\hat{z}^\prime$ is the opposite map of $\hat{z}^\prime$.
Two pairs $(z,\hat{z}),(z^\prime,\hat{z}^\prime)$ are \textit{equivalent} if $z=z^\prime$ and $\hat{z},\hat{z}^\prime$ are equivalent as cappings of $z=z^\prime$.
Let $[z,\hat{z}]$ denote the class represented by $(z,\hat{z})$ and define
\[\tilde{\Omega}_0(L)=\{[z,\hat{z}]; z\in\Omega_0(L), \text{ $\hat{z}$ is a capping of $z$}\}.\]

For a Hamiltonian function $H\colon S^1\times M\to\mathbb{R}$, the action functional $\mathcal{A}_H^L\colon \tilde{\Omega}_0(L)\to\mathbb{R}$ is defined by
\[\mathcal{A}_H^L([z,\hat{z}])=\int_0^1H_t(z(t))dt-\int_{D^2}\hat{z}^\ast\omega.\]

For a connected subset $X$ of $M$, let $\mathsf{Ch}(H;X)$ denote $\mathsf{Ch}(H;X,X,0_X)$.
We define its covering $\tilde{\mathsf{Ch}}(H;X)$ by
\[\tilde{\mathsf{Ch}}(H;X)=\{[z,\hat{z}]\in\tilde{\Omega}_0(X); z\in\mathsf{Ch}(H;X)\}.\]

Then we regard $\tilde{\mathsf{Ch}}(H;L)$ as the set of critical points of $\mathcal{A}_H^L$.
We define its spectrum $\mathrm{Spec}(H;L)$ as the set of critical values of $\mathcal{A}_H^L$ \textit{i.e.}
\[\mathrm{Spec}(H;L)=\{\mathcal{A}_H^L([z,\hat{z}]); [z,\hat{z}]\in\tilde{\mathsf{Ch}}(H;L)\}.\]

We define the \textit{non-degeneracy} of Hamiltonian functions as follows:
\begin{defi}
A Hamiltonian function $H\colon S^1\times M\to\mathbb{R}$ is called \textit{non-degenerate} if for any element  $z$ of $\mathsf{Ch}(H;L)$, the induced linear map $(d\phi_H^1)_{z(0)}\colon T_{z(0)}M\to T_{z(1)}M$ maps the linear subspace $T_{z(0)}L$ to a linear subspace transverse to $T_{z(1)}L$.
\end{defi}

When $H$ is non-degenerate, the Floer chain complex $CF_\ast(H;L)$ is generated by $\widetilde{\mathsf{Ch}}(H;L)$ as a module over $\mathbb{Z}/2\mathbb{Z}$.
The complex $CF_\ast(H;L)$ is graded by the Conley-Zehnder index $\operatorname{ind}_{CZ}$.
Note that $\operatorname{ind}_{CZ}([z,\hat{z}\natural A])=\operatorname{ind}_{CZ}([z,\hat{z}])-\mu(A)$ for any map $A\in \pi_2(M)$ in our convention.
Let $F\colon M\to\mathbb{R}$ be a function which is  the pullback of a Morse function $\tilde{F}\colon L\to\mathbb{R}$ on some Weinstein neighborhood of $L$ and $x$ a critical point of $\tilde{F}$  such that $d\tilde{F}$ is $C^1$-small near $x$.
Then $\operatorname{ind}_{Morse}(x)=\operatorname{ind}_{CZ}([x,c_x])$,
where $c_x$ is a trivial capping disk and $\operatorname{ind}_{Morse}$ is the Morse index.
We formally obtain the boundary map of this complex by counting isolated negative gradient flow lines of $\mathcal{A}_H^L$ and we define its homology group $HF_\ast(H;L)$ which is called \textit{the Lagrangian Floer homology on contractible trajectories} of $H$.

Oh \cite{Oh94}, Biran and Cornea \cite{BC} (see also \cite{LZ}) defined the quantum homology $QH_\ast(L)$ of a monotone Lagrangian submanifold $L$ with $N_L\geq2$ and proved that there exists a natural isomorphism $\Phi\colon QH_\ast(L)\to HF_\ast(H;L)$.
We call this isomorphism the PSS isomorphism (\cite{PSS}).

Biran and Cornea defined the quantum product $\ast$ on $QH_\ast(L)$.
$QH_\ast(L)$ has the fundamental class $[L]$ which is the unit with respect to $\ast$.
If $QH_\ast(L)\neq0$, then $[L]\neq0$ since $[L]$ is the unit of $QH_\ast(L)$.

Given an element $A=\sum_i a_i[z_i,\hat{z}_i]$ of $CF_\ast(H;L)$, we define the action level $l_H(A)$ of $A$ by
\[l_H(A)=\max\{\mathcal{A}_H^L([z_i,\hat{z}_i]);a_i\neq 0\}.\]

For a non-zero element $\mathtt{a}$ of $QH_\ast(L)$, we define the spectral invariant associated with $H$ and $\mathtt{a}$ by
\[c^L(\mathtt{a},H)=\inf\{l_H(A);[A]=\Phi(\mathtt{a})\}.\]
It is known that $c^L(\mathtt{a},H)$ is a finite number for any non-degenerate $H$ and any non-trivial $\mathtt{a}\in QH_\ast(L)$.

Let $H\colon S^1\times M\to\mathbb{R}$ be a continuous function, $\mathtt{a}$ a non-zero element of $QH_\ast(L)$.
We define the spectral invariant $c^L(\mathtt{a},H)$ associated with $H$ by
\[c^L(\mathtt{a},H)=\lim_{n\to\infty}c^L(\mathtt{a},H_n),\]
where $\{H_n\}_n$ is a sequence of non-degenerate Hamiltonian functions converging to $H$ in $L^\infty$-norm.
It is known that this limit exists and does not depend on the choice of $\{H_n\}_n$.
For a more precise argument, see Subsection 3.3 of \cite{LZ}.

In the present paper, the following proposition of Leclercq and Zapolsky plays a very important role.

\begin{prop}[Theorem 3 and Proposition 4 of \cite{LZ}]\label{LZ}
The spectral invariant has the following properties for any non-trivial elements $\mathtt{a},\mathtt{b}$ of $QH_\ast(L)$.

\begin{description}
\item[(1)\textit{Lipschitz continuity}]
For any continuous functions $F,G\colon S^1\times M\to\mathbb{R}$,
\[\int_0^1\min_M(F_t-G_t)dt\leq c^L(\mathtt{a},F)-c^L(\mathtt{a},G)\leq \int_0^1\max_M(F_t-G_t)dt,\]

\item[(2)\textit{Homotopy invariance}] 
Assume that normalized smooth functions $F,G\colon$ $S^1\times M\to\mathbb{R}$ satisfy $\phi^1_F=\phi^1_G$ and that their Hamiltonian isotopies $\{\phi^t_F\}$ and $\{\phi^t_G\}$ are homotopic relative to endpoints. Then $c^L(\mathtt{a},F)=c^L(\mathtt{a},G)$,

\item[(3)\textit{Triangle inequality}]
$c^L(\mathtt{a}\ast{\mathtt{b}},F\natural G) \leq c^L(\mathtt{a},F)+c^L(\mathtt{b},G)$ for any continuous functions $F,G\colon S^1\times M\to\mathbb{R}$,

\item[(4)\textit{Hamiltonian shift}]
Let $\rho\colon S^1\to\mathbb{R}$ be a continuous function.
Then, for any continuous function $H\colon S^1\times M\to\mathbb{R}$,
\[c^L(\mathtt{a},H+\rho)=c^L(\mathtt{a},H)+\int_0^1\rho(t)dt,\]
where $H+\rho\colon S^1\times M\to \mathbb{R}$ is a Hamiltonian function defined by $(H+\rho)(t,x)=H(t,x)+\rho(t)$,

\item[(5)]
$c^L([L],0)=0$,

\item[(6)\textit{Spectrality}]
For any smooth function $H\colon S^1\times M\to\mathbb{R}$,
\[c^L(\mathtt{a},H)\in\mathrm{Spec}(H;L).\]

\end{description}
\end{prop}

For a general Hamiltonian function $H\colon S^1\times M\to\mathbb{R}$ ($H$ can be degenerate), we define the spectral invariant $c^L(a,H)$ by the Lipschitz continuity.
Then the spectral invariant defined for general Hamiltonian functions also satisfies the properties in Proposition \ref{LZ}.

In this paper, we also consider spectral invariant defined in terms of the Hamiltnian Floer theory (see \cite{Sch}, \cite{Oh05b} and \cite{Oh}).
For a closed symplectic manifold $(M,\omega)$, we have the quantum homology $QH_\ast(M)$
(For simplicity, we consider the coefficient $\mathbb{Z}/2\mathbb{Z}$).
For a non-trivial element $\mathtt{a}$ of the quantum homology $QH_\ast(M)$ and a continuous function $H\colon S^1\times M\to\mathbb{R}$, let $c(\mathtt{a},H)$ denote the Oh-Schwarz (Hamiltonian) spectral invariant associated with $\mathtt{a}$ and $H$.
There are some conventions of the Oh-Schwarz spectral invariants, We use the same convention as \cite{EP09}.
In this paper, we use the same convention as \cite{EP09} and \cite{LZ} and compare our spectral invariants with the Oh-Schwarz spectral invariants.

The quantum homology $QH_\ast(M)$ acts on $QH_\ast(L)$ via the quantum module action (\cite{BC}, \cite{Z}),
\[\bullet\colon QH_\ast(M)\otimes QH_\ast(L)\to QH_\ast(L).\]
 Related to this action, we have the following convenient inequality.

\begin{prop}[Proposition 5 of \cite{LZ}]\label{ham and lagr}
Let $\mathtt{e}$ be an element of $QH_\ast(M)$ and $\mathtt{e}^\prime$ an element of $QH_\ast(L)$.

 For any Hamiltonian functions $F,G\colon S^1\times M\to\mathbb{R}$,
 \[c^L(\mathtt{e}\bullet\mathtt{e}^\prime,F\natural G)\leq c^L(\mathtt{e}^\prime,F)+c(\mathtt{e},G).\]
\end{prop}

\section{Proof of Corollary \ref{module find 1}}\label{module section}
We provide the following theorem which generalizes Corollary \ref{module find 1}.

\begin{thm}\label{module structure apply}
Let $L$ be a monotone Lagrangian submanifold of a closed symplectic manifold $(M,\omega)$ with $N_L\geq2$ and $QH_\ast(L)\neq0$.
Let $\mathtt{e}$ be an idempotent of $QH_\ast(M)$ and $\mathtt{e}^\prime$ an idempotent of $QH_\ast(L)$.
Assume that there exists an element $\mathtt{e}^{\prime\prime}$ of $QH_\ast(L)$ such that $\mathtt{e}\bullet\mathtt{e}^{\prime\prime}=\mathtt{e}^{\prime}$.
Then 
\begin{itemize}
\item
Any $\mathtt{e}$-superheavy subset $X$ is $(L,\mathtt{e}^{\prime})$-superheavy,
\item
Any $(L,\mathtt{e}^\prime)$-heavy subset $Y$ is $\mathtt{e}$-heavy.
\end{itemize}
\end{thm}

Since $[M]\bullet e^\prime=e^\prime$, Corollary \ref{module find 1} follows from Theorem \ref{module structure apply}.

To prove Theorem \ref{module structure apply}, we prove the following lemma.

\begin{lem}\label{module structure lemma}
Let $L$ be a monotone Lagrangian submanifold of a closed symplectic manifold $(M,\omega)$ with $N_L\geq2$ and $QH_\ast(L)\neq0$.
Let $\mathtt{e}$ be an idempotent of $QH_\ast(M)$ and $\mathtt{e}^\prime$ an idempotent of $QH_\ast(L)$.
Assume that there exists an element $\mathtt{e}^{\prime\prime}$ of $QH_\ast(L)$ such that $\mathtt{e}\bullet\mathtt{e}^{\prime\prime}=\mathtt{e}^{\prime}$.
Then 
\[\zeta^L_{\mathtt{e^\prime}}(F)\leq \zeta_{\mathtt{e}}(F).\]
\end{lem}

\begin{proof}
By Proposition \ref{ham and lagr}, for any integer $k$,
\[c^L(\mathtt{e}^{\prime},kF)\leq c(\mathtt{e},kF)+c^L(\mathtt{e}^{\prime\prime},0).\]
Thus, by diving by $k$ and take the limit, we complete the proof.
\end{proof}

By Definition \ref{definition of heavy}, Theorem \ref{module structure apply} immediately follows from Lemma \ref{module structure lemma}.
The idea of comparing Lagrangian spectral invariant and Hamiltonian spectral invariant using the module structure $\bullet$ comes from Section 6 of \cite{LZ}.

\section{The bounded spectrum condition and Lagrangian version of Ostrover's trick}\label{Ostrover section}
The following notion was introduced by the author.
\begin{defi}[\cite{K14}]\label{bounded spectrum}
Let $L$ be a monotone Lagrangian submanifold of a closed symplectic manifold $(M,\omega)$ with $N_L\geq2$ and $QH_\ast(L)\neq0$.
Let $\mathtt{a}$ be an idempotent of $QH_\ast(L)$.

An open subset $U$ of $M$ satisfies \textit{the bounded spectrum condition} with respect to $\mathtt{a}$ if there is a positive number $K$ such that
\[c^L(\mathtt{a},F)<K,\]
for any Hamiltonian function $F\colon S^1\times U\to\mathbb{R}$ with compact support.

A subset $X$ of $M$  satisfies \textit{the bounded spectrum condition} with respect to $\mathtt{a}$ if there is an open neighborhood $U$ of $X$ such that $U$ satisfies the bounded spectrum condition with respect to $\mathtt{a}$.
\end{defi}

In this section, we give some examples of open subsets with the bounded spectrum condition.

By considering a Lagrangian analogue of Ostrover's argument \cite{O}, we prove the following proposition.
We use this proposition for several times in this paper.

\begin{prop}\label{Ostrover}
Let $L$ be a monotone Lagrangian submanifold of a closed symplectic manifold $(M,\omega)$ with $N_L\geq2$ and $QH_\ast(L)\neq0$.
Let $U$ be an open subset of $M$ and $H$ a Hamiltonian function satisfying $\phi_H(U)\cap L=\emptyset$.
Then for any idempotent $\mathtt{a}$ of the quantum homology $QH_\ast(L)$ and  any Hamiltonian function $F\colon S^1 \times U\to\mathbb{R}$ with compact support on $S^1\times U$,
\[c^L(\mathtt{a},F)\leq c^L(\mathtt{a},H)+c^L(\mathtt{a},\bar{H}).\]
\end{prop}

\begin{proof}

We give an upper bound of the spectral invariant associated with $F$ using the concatenation with $\phi_H^t$.

We can choose a smooth function $\chi\colon[0,\frac{1}{2}]\to[0,1]$ satisfying the following conditions.
\begin{itemize}
\item $\frac{\partial\chi}{\partial t}(t)\geq0$ for any $t\in{[0,\frac{1}{2}]}$, and
\item $\chi(t)=0$ for any $t\in[0,\frac{1}{5}]$ and $\chi(t)=1$ for any $t\in{[\frac{2}{5},\frac{1}{2}]}$.
\end{itemize}

For a real number $s$ with $s \in [0,1]$, we define the new Hamiltonian function $K^s \colon S^1\times M\to\mathbb{R}$ as follows:
\begin{equation*}
K^s(t,x)=
\begin{cases}
\frac{\partial\chi}{\partial t}(t)\bar{H}(\chi(t),x) & \text{when }t\in[0,\frac{1}{2}], \\
s\frac{\partial\chi}{\partial t}(t-\frac{1}{2})F(s\chi(t-\frac{1}{2}),x)& \text{when }t\in[\frac{1}{2},1].
\end{cases}
\end{equation*}
Since $\frac{\partial\chi}{\partial t}=0$ on neighborhoods of $t=0$ and $t=\frac{1}{2}$, $K^s$ is a smooth Hamiltonian function.

We claim  $\operatorname{Spec}(K^s;L)\subset\operatorname{Spec}(\bar{H};L)$ for a real number $s$ with $s \in [0,1]$.
Let $F^s\colon S^1\times\hat{W}\to\mathbb{R}$ denote the Hamiltonian function defined by 
\[F^s(t,x)=s\frac{\partial\chi}{\partial t}(\frac{t}{2})F(s\chi(\frac{t}{2}),x).\]
Let $[z,\hat{z}]$ be an element of $\tilde{\mathsf{Ch}}(\bar{H};L)$ and define $x\in M$ by $x=z(0)$.
If $x\in \phi_H(U)$, then $\gamma_{K^s}^x\notin\Omega_0(L)$ and in particular,  $\gamma_{K^s}^x\notin\mathsf{Ch}(\bar{H};L)$ since $\phi_H(U)\cap L=\emptyset$.
If $x\notin \phi_H(U)$, then $\phi_{\bar{H}}(x)\notin U$.
Thus $\gamma_{K^s}^x$ is equal to $\gamma_{\bar{H}}^x$ up to parameter change and $\int_0^1H(t,\gamma_{\bar{H}}^x(t))dt=\int_0^1K^s(t,\gamma_{K^s}^x(t))dt$.
Therefore we see that there exists a natural inclusion map $\iota\colon  \tilde{\mathsf{Ch}}(K^s;L)\to \tilde{\mathsf{Ch}}(\bar{H};L)$ which preserves values of the action functional, and hence $\operatorname{Spec}(K^s;L)\subset\operatorname{Spec}(\bar{H};L)$.
implies $c^L(\mathtt{a},K^s)\in\operatorname{Spec}(\bar{H};L)$.

By the Lipschitz continuity for spectral invariants (Proposition \ref{LZ} (1)), $c^L(\mathtt{a},K^s)$ depends continuously on $s$.
Since $\operatorname{Spec}(\bar{H};L)$ is a measure-zero set (Lemma 30 of \cite{LZ}),  $c^L(\mathtt{a},K^s)$ is a constant function of $s$.
The homotopy invariance for spectral invariants (Proposition \ref{LZ} (2)) implies 
\[c^L(\mathtt{a},K^0)=c^L(\mathtt{a},\bar{H}).\]
Hence for any $s \in [0,1]$,
 \[c^L(\mathtt{a},K^s)=c^L(\mathtt{a},\bar{H}).\]
Then $c^L(\mathtt{a},F)$ is estimated as follows.
 \begin{align*}
c^L(\mathtt{a},F)& \leq c^L(\mathtt{a},K^1)+c^L(\mathtt{a},H)\\
& = c^L(\mathtt{a},\bar{H})+c^L(\mathtt{a},H).
\end{align*}

\end{proof}



For (abstractly) displaceable subsets, we have the following proposition.

\begin{prop}\label{abstract disp bounded spec ham 1}
Let $L$ be a monotone Lagrangian submanifold of a closed symplectic manifold $(M,\omega)$ with $N_L\geq2$ and $QH_\ast(L)\neq0$.

Let $U$ be an open subset of $M$ and $H$ a Hamiltonian function satisfying $\phi_H(U)\cap \bar{U}=\emptyset$.
Then for any idempotent $\mathtt{a}$ of the quantum homology $QH_\ast(L)$ and  any Hamiltonian function $F\colon S^1 \times U\to\mathbb{R}$ with compact support on $S^1\times U$,
\[c^L(\mathtt{a},F)\leq c([M],H)+c([M],\bar{H})+c^L(\mathtt{a},0).\]
\end{prop}

To prove Proposition \ref{abstract disp bounded spec ham 1},
we recall the original Ostrover's trick.

\begin{prop}[Proposition 3.1 of \cite{U}, see also Proposition 2.1 of \cite{FGS}]\label{abstract disp bounded spec ham 2}
Let $(M,\omega)$ be a closed symplectic manifold.
Let $U$ be an open subset of $M$ and $H$ a Hamiltonian function satisfying $\phi_H(U)\cap \bar{U}=\emptyset$.
Then for any idempotent $\mathtt{a}$ of the quantum homology $QH_\ast(M,\omega)$ and  any Hamiltonian function $F\colon S^1 \times U\to\mathbb{R}$ with compact support on $S^1\times U$,
\[c(\mathtt{a},F)\leq c(\mathtt{a},H)+c(\mathtt{a},\bar{H}).\]
\end{prop}

\begin{proof}[Proof of Proposition \ref{abstract disp bounded spec ham 1}]
By Proposition \ref{ham and lagr},
\[c^L(\mathtt{a},F)\leq c([M],F)+c^L(\mathtt{a},0).\]
By Proposition \ref{abstract disp bounded spec ham 2},
\[c([M],F)\leq c([M],H)+c([M],\bar{H}).\]
Thus,
\[c^L(\mathtt{a},F)\leq c([M],H)+c([M],\bar{H})+c^L(\mathtt{a},0).\]
\end{proof}

By the Lipschitz continuity of Lagrangian spectral invariant, the following corollary immediately follows from Proposition \ref{Ostrover}, \ref{abstract disp bounded spec ham 1} and (1), (5) of Proposition \ref{LZ}.

\begin{cor}\label{Hofer Ostrover}
Let $L$ be a monotone Lagrangian submanifold of a closed symplectic manifold $(M,\omega)$ with $N_L\geq2$ and $QH_\ast(L)\neq0$.
Let $U$ be an open subset of $M$ which is displaceable from $L$ or $U$ itself.
Then for   any Hamiltonian function $F\colon S^1 \times U\to\mathbb{R}$ with compact support on $S^1\times U$,
\[c^L([L],F)\leq \bar{E}(U;L).\]
\end{cor}




\section{Application to Poisson bracket invariant}\label{pb section}

\begin{lem}\label{sum composition deiff}

For any Hamiltonian functions $F,G\colon S^1\times M\to\mathbb{R}$ and a non-trivial element $a$ of $QH_\ast(L)$,

\[|c^L(\mathtt{a},F+ G)-c^L(\mathtt{a},F\natural G)|\leq \frac{1}{2}||\{F,G\}||.\]
\end{lem}

\begin{proof}

For any $t>0$ and $x\in M$,
\[G(\phi_F^tx)-G(x)=\int_0^t\frac{d}{ds}G(\phi_F^sx)ds=\int_0^t\{G,F\}(\phi_F^sx)ds\]
Thus, for any $t$,
\[||(F+G)-(F\natural G)_t||=||G-G\circ(\phi_F^t)^{-1}||=||G\circ\phi_F^t-G||\leq t||\{F,G\}||.\]

By the Lipshitz continuity,
\begin{align*}
&c^L(\mathtt{a},F+G)-c^L(\mathtt{a},F\natural G)\\
      &\leq\int_0^1||(F+G)-(F\natural G)_t||dt\\
      &\leq\int_0^1t||\{F,G\}||dt\\
      & = \frac{1}{2}||\{F,G\}||.
\end{align*}

\end{proof}

\begin{proof}[Proof of Theorem \ref{positive poisson}]
Let $\vec{F}$ be a partition of the unity subordinated to $\mathcal{U}$.

Fix a positive number $R$.
Define functions $G_k\colon M\to\mathbb{R}$ ($k=0,\ldots,N$), $H_k\colon \mathbb{R}\to\mathbb{R}$ $(k=1,\ldots,N)$ by $G_k(x)=\sum_{i=1}^kF_i(x)$ and $H_k(R)=c^L([L],RG_k)-c^L([L],RG_{k-1})$. 

Then, by the triangle inequality,
\[c^L([L],R\bar{F}_k)\leq c^L([L],(RG_{k-1})\natural (RF_k))-c^L([L],RG_{k-1})\leq c^L([L],RF_k).\]

Since $\mathrm{Supp}(F_k)\subset U_k$ for any $k=1,\ldots,N$, by Corollary \ref{Hofer Ostrover},
\[-\bar{E}(\mathcal{U};L)\leq c^L([L],(RG_{k-1})\natural (RF_k))-c^L([L],RG_{k-1})\leq \bar{E}(\mathcal{U};L).\]
Hence $|c^L([L],(RG_{k-1})\natural (RF_k))-c^L([L],RG_{k-1})|\leq \bar{E}(\mathcal{U};L)$.

 By Lemma \ref{sum composition deiff} and the definition of $\kappa_{cl}(\vec{F})$,
\begin{align*}
 &|c^L([L],RG_{k-1}+RF_k)-c^L([L],(RG_{k-1})\natural (RF_k))|\\
      &\leq\frac{1}{2}||\{RG_{k-1},RF_k\}||\\
      &=\frac{R^2}{2}||\{G_{k-1},F_k\}||\\
      & \leq \frac{R^2}{2} \kappa_{cl}(\vec{F}).
\end{align*}

Thus, by the triangle inequality,
\begin{align*}
|H_k(R)|&=|c^L([L],RG_k)-c^L([L],RG_{k-1})|\\
      &\leq |c^L([L],RG_{k-1}+RF_k)-c^L([L],(RG_{k-1})\natural (RF_k))|\\
      &+|c^L([L],(RG_{k-1})\natural (RF_k))-c^L([L],RG_{k-1})|\\
      & \leq \frac{R^2}{2} \kappa_{cl}(\vec{F})+\bar{E}(\mathcal{U};L).
\end{align*}

By the triangle inequality,
\begin{align*}
&c([L],R)-c([L],0)\\
      &=H_N(R)+H_{N-1}(R)+\cdots+H_1(R)\\
      & \leq |H_N(R)|+|H_{N-1}(R)|+\cdots+|H_1(R)|\\
      & \leq N(\frac{R^2}{2} \kappa_{cl}(\vec{F})+\bar{E}(\mathcal{U};L)).
\end{align*}
By the Hamiltonian shift property ((4) of Proposition \ref{LZ}), $c([L],R)-c([L],0)=R$.
Thus
\[1\leq N(\frac{R}{2} \kappa_{cl}(\vec{F})+R^{-1}\cdot \bar{E}(\mathcal{U};L)).\]
The right-hand side is minimized by $R=(2\bar{E}(\mathcal{U};L))^{\frac{1}{2}}(\kappa_{cl}(\vec{F}))^{-\frac{1}{2}}$ and hence
\[1\leq N(2\bar{E}(\mathcal{U};L)\cdot\kappa_{cl}(\vec{F}))^{\frac{1}{2}}.\]
Therefore
\[ \kappa_{cl}(\vec{F})\cdot\bar{E}(\mathcal{U};L)\geq (2N^2)^{-1}.\]
By taking the infimum over all partition $\vec{F}$ of unity subordinated to $\mathcal{U}$,
\[\mathrm{pb}(\mathcal{U})\cdot \bar{E}(\mathcal{U};L)\geq (2N^2)^{-1}.\]
\end{proof}

\begin{proof}[Proof of Proposition \ref{disp pb inv}]
Let $F$ be a smooth function such that $F^{-1}(0)=X$.
Since $F^{-1}(0)=X$ is displaceable from $X$, there is a positive number $\epsilon$ such that $F^{-1}((-2\epsilon,2\epsilon))$ is also displaceable from $X$.
Define open subsets $U_1,U_2,U_3$ by $U_1=F^{-1}((-\infty,-\epsilon))$ $U_2=F^{-1}((-2\epsilon,2\epsilon))$ $U_3=F^{-1}((\epsilon,+\infty))$ and set $\mathcal{U}=\{U_1,U_2,U_3\}$.
By the definition, $U_2$ is displaceable from $X$.
$U_1$ and $U_3$ are disjoint from $X$ and in particular, displaceable from $X$.
Since $U_1\cap U_2\cap U_3=\emptyset$, $\mathrm{pb}(\mathcal{U})=0$.
Thus this $\mathcal{U}$ satisfies all of the conditions.
\end{proof}

\section{Properties of partial quasi-state}\label{Properties of partial quasi-state}


Let $L$ be a monotone Lagrangian submanifold of a closed symplectic manifold $(M,\omega)$ with $N_L\geq 2$ and $QH_\ast(L)\neq0$.

$\zeta_\mathtt{a}^L$ satisfies the following properties.

\begin{prop}\label{quasi-state basic}

\begin{description}
\item[(1)\textit{Partial quasi-additivity}]
For Hamiltonian functions $F,G\colon M\to\mathbb{R}$ with $\{F,G\}=0$ and $\mathrm{Supp}(G)$ is displaceable from $\mathrm{Supp}(G)$ itself,
$\zeta_\mathtt{a}^L(F+G)=\zeta_\mathtt{a}^L(F)$.

\item[($1^\prime$)\textit{Lagrangian partial quasi-additivity}]
For Hamiltonian functions $F,G\colon M\to\mathbb{R}$ with $\{F,G\}=0$ and $\mathrm{Supp}(G)$ is displaceable from $L$,
$\zeta_\mathtt{a}^L(F+G)=\zeta_\mathtt{a}^L(F)$.

\item[(2)\textit{Normalization}]
$\zeta_\mathtt{a}^L(1)=1$.

\item[(3)\textit{Semi-homogeneity}]
For any autonomous Hamiltonian function $F\colon M\to\mathbb{R}$ and any positive number $s$,
$\zeta_\mathtt{a}^L(sF)=s\zeta_\mathtt{a}^L(F)$.

\end{description}
\end{prop}
We prove (1) of Proposition \ref{quasi-state basic} and ($1^\prime$) in Section \ref{stem section}.
Th proofs of (2) and (3) of Proposition \ref{quasi-state basic} are quite similar to the corresponding statements on Oh-Schwarz spectral invariants (Theorem 3.6 of \cite{EP09}) and thus we omit the proofs.

To prove Theorem \ref{hv from shv}, we prepare some properties of $\zeta_{\mathtt{a}}^L$.
\begin{prop}\label{hv shv intersect}
Let $L$ be a monotone Lagrangian submanifold of a closed symplectic manifold $(M,\omega)$ with $N_L\geq2$ and $QH_\ast(L)\neq0$ and $\mathtt{a}$ a non-trivial idempotent of $QH_\ast(L)$.
Let $X$, $Y$ be a $(L,\mathtt{a})$-heavy subset, a $(L,\mathtt{a})$-superheavy subset of $M$, respectively.
Then $X\cap Y\neq \emptyset$.
\end{prop}

The proof of Proposition \ref{hv shv intersect} is quite similar to the proof of (iii) of Theorem 1.4 in \cite{EP09} and thus we omit the proof.

\begin{lem}\label{conj invariance}
For any continuous function $F\colon S^1\times M\to\mathbb{R}$ and a Hamiltonian diffeomorphism $\psi$,
\[\zeta^L_{\mathtt{a}}(F\circ\psi)=\zeta^L_{\mathtt{a}}(F),\]
where $F\circ\psi$ is a function defined by
\[F\circ\psi(t,x)=F(t,\psi(x)).\]
\end{lem}

The proof of Proposition \ref{conj invariance} is quite similar to the proof of 2 of Theorem 1.8 in \cite{MVZ} and thus we omit the proof.

By the definition of heaviness and superheaviness,
Lemma \ref{conj invariance} implies the following corollary.

\begin{cor}\label{hv moved}
Let $L$ be a monotone Lagrangian submanifold of a closed symplectic manifold $(M,\omega)$ with $N_L\geq2$ and $QH_\ast(L)\neq0$, $\mathtt{a}$ a non-trivial idempotent of $QH_\ast(L)$ and $X$ a $(L,\mathtt{a})$-heavy subset.
Then $\psi(X)$ is a $(L,\mathtt{a})$-heavy subset for any Hamiltonian diffeomorphism $\psi$.
\end{cor}

\begin{proof}[Proof of Theorem \ref{hv from shv}]
To prove by contradiction, suppose that there exists a Hamiltonian diffeomorphism $\psi$ such that
\[\psi(X)\cap Y=\emptyset.\]
By Corollary \ref{hv moved}, $\psi(X)$ is $(L,\mathtt{a})$-heavy.
Since $Y$ is $(L,\mathtt{a})$-superheavy, this contradicts Proposition \ref{hv shv intersect}.

\end{proof}

\section{Poisson bracket inequality on  Lagrangian spectral invariants}\label{PB ineq section}
A large difference between our work and Entov and Polterovich's original argument is that
Lagrangian spectral invariants are not invariant under the action of Hamiltonian diffeomorphisms.
We need a more complicated argument than the Entov and Polterovich's original one.

In this section, we use spectral invariants defined on the universal covering of the group of Hamiltonian diffeomorphisms.
To define them, we prepare some notions.

Let $(M,\omega)$ be a symplectic manifold.
Let $\widetilde{\mathrm{Ham}}(M,\omega)$ be the universal covering of the group $\mathrm{Ham}(M,\omega)$ of Hamiltonian diffeomorphisms.
Note that an element of $\widetilde{\mathrm{Ham}}(M,\omega)$ is represented by a path in $\mathrm{Ham}(M,\omega)$ starting from the identity.
For a Hamiltonian function $H\colon S^1\times M\to\mathbb{R}$ with compact support, let $\tilde{\phi}_H$ denote an element of $\widetilde{\mathrm{Ham}}(M,\omega)$ represented by the path $\{\phi_H^t\}_{t\in[0,1]}$ in $\mathrm{Ham}(M,\omega)$.

For a non-trivial element $\mathtt{a}$ of $QH_\ast(L)$ and an element $\tilde{\phi}$ of $\widetilde{\mathrm{Ham}}(M,\omega)$, we define its spectral invariant $c^L(\mathtt{a},\tilde{\phi})$ by
\[c^L(\mathtt{a},\tilde{\phi})=c^L(\mathtt{a},H),\]
where $H\colon S^1\times M\to\mathbb{R}$ is a normalized Hamiltonian function with compact support such that $\tilde{\phi}_H=\tilde{\phi}$.
By the homotopy invariance ((2) of Proposition \ref{LZ}),
$c^L(\mathtt{a},H)$ does not depend on the choice of $H$ and hence $c^L(\mathtt{a},\tilde{\phi})$ is well-defined.
For normalized Hamiltonian functions $F,G\colon S^1\times M\to\mathbb{R}$, $F\natural G$ is also a normalized Hamiltonian function.
Thus, the following triangle inequality follows from (3) of Proposition \ref{LZ}.
\begin{prop}\label{diffeo tri}
$c^L(\mathtt{a}\ast{\mathtt{b}},\tilde{\phi}\tilde{\psi}) \leq c^L(\mathtt{a},\tilde{\phi})+c^L(\mathtt{b},\tilde{\psi})$ for any  $\tilde{\phi},\tilde{\psi}\in\widetilde{\mathrm{Ham}}(M,\omega)$ and any $\mathtt{a},\mathtt{b}\in QH_\ast(L)$.
\end{prop}

For a non-trivial idempotent $\mathtt{a}$ of $QH_\ast(L)$, we define the homogenization $\sigma_{\mathtt{a}}^L\colon$ $\widetilde{\mathrm{Ham}}(M,\omega)\to\mathbb{R}$ of $c^L(\mathtt{a},\cdot)$ by
\[\sigma_{\mathtt{a}}^L(\tilde{\phi})=\lim_{k\to+\infty}\frac{c^L(\mathtt{a},\tilde{\phi}^k)}{k}.\]
By Proposition \ref{diffeo tri}, we can prove the existence of this limit.

For a non-trivial element $\mathtt{a}$ of $QH_\ast(L)$ and elements $\tilde{f},\tilde{g}$ of $\widetilde{\mathrm{Ham}}(M,\omega)$, we define the following invariant $q_{\mathtt{a},\tilde{f}}(\tilde{g})$ by
\[q_{\mathtt{a},\tilde{f}}(\tilde{g})=\sup_{k\in\mathbb{Z}}\{c^L(\mathtt{a},\tilde{f}^{-k}\tilde{g} \tilde{f}^k)\}+\sup_{k\in\mathbb{Z}}\{c^L(\mathtt{a},\tilde{f}^{-k}\tilde{g}^{-1}\tilde{f}^k)\}.\]
For simplicity, let $q_{\mathtt{a}}(\tilde{g})$ denote $q_{\mathtt{a},\mathrm{id}}(\tilde{g})$.
The author does not know whether $q_{\mathtt{a},\tilde{f}}(\tilde{g})<+\infty$ for any $\tilde{f},\tilde{g}\in\widetilde{\mathrm{Ham}}(M,\omega)$.

Then we can prove the following proposition which generalizes Proposition 3.5.3 in \cite{PR}.

\begin{prop}\label{iroirodasu lemma}
Let $L$ be a monotone Lagrangian submanifold of a closed symplectic manifold $(M,\omega)$ with $N_L\geq2$ and $QH_\ast(L)\neq0$.

Then, for any non-trivial idempotent $\mathtt{a}$ of $QH_\ast(L)$ and any $\tilde{f},\tilde{g}\in\widetilde{\mathrm{Ham}}(M,\omega)$,
\[|\sigma_{\mathtt{a}}^L(\tilde{f}\tilde{g})-\sigma_{\mathtt{a}}^L(\tilde{f})-\sigma_{\mathtt{a}}^L(\tilde{g})|\leq q_{\mathtt{a},\tilde{f}}(\tilde{g}).\]
In particular, $q_{\mathtt{a},\tilde{f}}(\tilde{g})\geq0$.
\end{prop}


\begin{proof}

Note that
\[(\tilde{f}\tilde{g})^k=\tilde{h}\tilde{f}^k\text{, where }\tilde{h}=(\tilde{f}\tilde{g}\tilde{f}^{-1})(\tilde{f}^2\tilde{g}\tilde{f}^{-2})\cdots(\tilde{f}^k\tilde{g}\tilde{f}^{-k}).\]
Thus, by Proposition \ref{diffeo tri},
\begin{align*}
&c^L(\mathtt{a},(\tilde{f}\tilde{g})^k)\leq c^L(\mathtt{a},\tilde{h})+c^L(\mathtt{a},\tilde{f}^k),\\
&c^L(\mathtt{a},(\tilde{f}\tilde{g})^k)\geq c^L(\mathtt{a},\tilde{f}^k)-c^L(\mathtt{a},\tilde{h}^{-1}).
\end{align*}
Thus
\[-c^L(\mathtt{a},\tilde{h}^{-1})-c^L(\mathtt{a},\tilde{g}^k)\leq c^L(\mathtt{a},(\tilde{f}\tilde{g})^k)-c^L(\mathtt{a},\tilde{f}^k)-c^L(\mathtt{a},\tilde{g}^k)\leq c^L(\mathtt{a},\tilde{h})-c^L(\mathtt{a},\tilde{g}^k).\]



Since $\tilde{h}^{-1}=(\tilde{f}^{k}\tilde{g}^{-1}\tilde{f}^{-k})\cdots(\tilde{f}^{2}\tilde{g}^{-1}\tilde{f}^{-2})(\tilde{f}^1\tilde{g}^{-1}\tilde{f}^{-1})$, by Proposition \ref{diffeo tri} and the definition of $q_{\mathtt{a},\tilde{f}}(\tilde{g})$,
\begin{align*}
&c^L(\mathtt{a},\tilde{h}^{-1})+c^L(\mathtt{a},\tilde{g}^k)\\
&\leq c^L(\mathtt{a},\tilde{f}^{k}\tilde{g}^{-1}\tilde{f}^{-k})+\cdots+c^L(\mathtt{a},\tilde{f}^{2}\tilde{g}^{-1}\tilde{f}^{-2})+c^L(\mathtt{a},\tilde{f}^1\tilde{g}^{-1}\tilde{f}^{-1})+kc^L(\mathtt{a},\tilde{g})\\
&\leq \sum_{i=-k}^{-1}(c^L(\mathtt{a},\tilde{f}^{-i}\tilde{g}^{-1}\tilde{f}^i)+c^L(\mathtt{a},\tilde{g}))\\
&\leq kq_{\mathtt{a},\tilde{f}}(\tilde{g}).
\end{align*}
By Proposition \ref{diffeo tri},
\[c^L(\mathtt{a},\mathrm{id})-c^L(\mathtt{a},\tilde{g}^k)\leq c^L(\mathtt{a},\tilde{g}^{-k})\leq kc^L(\mathtt{a},\tilde{g}^{-1}) .\]
Hence $-c^L(\mathtt{a},\tilde{g}^k)\leq kc^L(\mathtt{a},\tilde{g}^{-1})-c^L(\mathtt{a},\mathrm{id})$.
Thus, by Proposition \ref{diffeo tri} and the definitions of $\tilde{h}$ and $q_{\mathtt{a},\tilde{f}}(\tilde{g})$.
\begin{align*}
&c^L(\mathtt{a},\tilde{h})-c^L(\mathtt{a},\tilde{g}^k)\\
&\leq c^L(\mathtt{a},\tilde{f}^{-1}\tilde{g}\tilde{f})+c^L(\mathtt{a},\tilde{f}^{-2}\tilde{g}\tilde{f}^2)+\cdots+c^L(\mathtt{a},\tilde{f}^{-k}\tilde{g}\tilde{f}^k)+kc^L(\mathtt{a},\tilde{g}^{-1})-c^L(\mathtt{a},\mathrm{id})\\
&\leq \sum_{i=1}^k(c^L(\mathtt{a},\tilde{f}^{-i}\tilde{g}\tilde{f}^i)+c^L(\mathtt{a},\tilde{g}^{-1}))-c^L(\mathtt{a},\mathrm{id})\\
&\leq kq_{\mathtt{a},\tilde{f}}(\tilde{g})-c^L(\mathtt{a},\mathrm{id}).
\end{align*}
Thus

\[-kq_{\mathtt{a},\tilde{f}}(\tilde{g})\leq c^L(\mathtt{a},(\tilde{f}\tilde{g})^k)-c^L(\mathtt{a},\tilde{f}^k)-c^L(\mathtt{a},\tilde{g}^k)\leq kq_{\mathtt{a},\tilde{f}}(\tilde{g})-c^L(\mathtt{a},\mathrm{id}).\]
By dividing by $k$ and passing to the limit as $k\to+\infty$, we complete the  proof.

\end{proof}

For autonomous Hamiltonian functions $F,G\colon M\to\mathbb{R}$, we define $D(F,G)$ by
\[D(F,G)=\min\{\sup_{t\in\mathbb{R}}q_{\mathtt{a},\tilde{\phi}_F}(\tilde{\phi}_{tG}),\sup_{t\in\mathbb{R}}q_{\mathtt{a},\tilde{\phi}_G}(\tilde{\phi}_{tF})\}.\]
By Proposition \ref{iroirodasu lemma}, $D(F,G)\geq0$ or $D(F,G)=+\infty$.

We give a generalization of Poisson bracket inequality.
\begin{prop}\label{pb ineq}
Let $L$ be a monotone Lagrangian submanifold of a closed symplectic manifold $(M,\omega)$ with $N_L\geq2$ and $QH_\ast(L)\neq0$.
Then, for any non-trivial idempotent $\mathtt{a}$ of $QH_\ast(L)$ and any Hamiltonian functions $F,G\colon M\to\mathbb{R}$ with $D(F,G)<+\infty$,
\[|\zeta_{\mathtt{a}}^L(F+G)-\zeta_{\mathtt{a}}^L(F)-\zeta_{\mathtt{a}}^L(G)|\leq (2D(F,G)\cdot ||\{F,G\}||)^{1/2}.\]
\end{prop}
In the present paper, we give good upper bounds of $D(F,G)$ under some situations.

By Proposition \ref{LZ} and Proposition \ref{iroirodasu lemma},
we can prove Proposition \ref{pb ineq} quite similarly to Proposition 4.6.1 of \cite{PR}.
Therefore we omit the proof of Proposition \ref{pb ineq}.

\begin{rem}\label{technical remark on pb}
Let $U$ be an open subset with the bounded spectrum condition.
If the equality ``$q_{\mathtt{a},\tilde{f}}(\tilde{g})=q_{\mathtt{a},\mathrm{id}}(\tilde{g})$'' holds for any $\tilde{f},\tilde{g}\in\widetilde{\mathrm{Ham}}(M,\omega)$,
then we have $D(F,G)<K$ for any Hamiltonian function  $F\colon M\to\mathbb{R}$ with $\mathrm{Supp}(F)\subset U$.
 
 In the case of Oh-Schwarz spectral invariants, a similar equality actually holds and hence we obtain an upper bound of $D(F,G)$.
 However, in our case, the above equality does not hold and thus it is more difficult to give an upper bound of $D(F,G)$.
 
\end{rem}

\section{Proof of partial quasi-additivities and Theorem \ref{stem shv}}\label{stem section}

\begin{prop}\label{vanishing property}
Let $G\colon M\to\mathbb{R}$ be a Hamiltonian function.
Assume that $\mathrm{Supp}(G)$ satisfies the bounded spectrum condition with respect to some idempotent $\mathtt{a}$ of $QH_\ast(L)$,
$\zeta_\mathtt{a}^L(G)=0$.
\end{prop}
\begin{proof}
By the definition of the bounded spetrum condition,  there is a positive number $K$ such that
\[c^L(\mathtt{a},kG)<K,\]
for any $k\in\mathbb{Z}$.
Then, by the triangle inequality and the homotopy invariance ((2) and (3) of Proposition \ref{LZ}),
\[c^L(\mathtt{a},kG)\geq -c^L(\mathtt{a},-kG)+c^L(\mathtt{a},0)>-K+c^L(\mathtt{a},0),\]
and hence 
\[-K+c^L(\mathtt{a},0)<c^L(\mathtt{a},kG)<K,\]
for any $k\in\mathbb{Z}$.
Thus
\[\zeta_\mathtt{a}^L(G)=\lim_{k\to+\infty}\frac{c^L(\mathtt{a},kG)}{k}=0.\]
\end{proof}

To prove (1) and ($1^\prime$) of Proposition \ref{quasi-state basic}, we give a more general proposition.
\begin{prop}\label{general P qa}
For Hamiltonian functions $F,G\colon M\to\mathbb{R}$ with $\{F,G\}=0$ and $\mathrm{Supp}(G)$ satisfies the bounded spectrum condition with respect to $\mathtt{a}$,
$\zeta_\mathtt{a}^L(F+G)=\zeta_\mathtt{a}^L(F)$.
\end{prop}

\begin{proof}
Since $\{F,G\}=0$, $(\tilde{\phi}_G)^{-k}\tilde{\phi}_F^t(\tilde{\phi}_G)^k=\tilde{\phi}_F^t$ for any $t\in\mathbb{R}$ and any $k\in\mathbb{Z}$.
Hence $q_{\mathtt{a},\tilde{\phi}_G}(\tilde{\phi}_F)=q_{\mathtt{a}}(\tilde{\phi}_F)$ and $D(F,G)<+\infty$.
Thus, by Proposition \ref{pb ineq},
\[||\zeta_{\mathtt{a}}^L(F+G)-\zeta_{\mathtt{a}}^L(F)-\zeta_{\mathtt{a}}^L(G)||\leq (2D(F,G)\cdot||\{F,G\}||)^{1/2}=0.\]
Hence, by Proposition \ref{vanishing property},
\[\zeta_{\mathtt{a}}^L(F+G)=\zeta_{\mathtt{a}}^L(F)+\zeta_{\mathtt{a}}^L(G)=\zeta_{\mathtt{a}}^L(F).\]
\end{proof}

\begin{proof}[Proof of ($1^\prime$) of Proposition \ref{quasi-state basic}]
Since $\mathrm{Supp}(G)$ is displaceable from $L$, there is an open neighborhood $U$ of $\mathrm{Supp}(G)$ which is displaceable from $L$.
Thus, by Proposition \ref{Ostrover}, $\mathrm{Supp}(G)$ satisfies the bounded spectral condition with respect to any idempotent $\mathtt{a}$ and thus, by Proposition \ref{general P qa},
\[\zeta_{\mathtt{a}}^L(F+G)=\zeta_{\mathtt{a}}^L(F).\]
\end{proof}
We can prove (1) of Proposition \ref{quasi-state basic} by Proposition \ref{abstract disp bounded spec ham 1} similarly to ($1^\prime$) of Proposition \ref{quasi-state basic}.

The proof of Theorem  \ref{stem shv} is quite similar to the one of the original stem case if we know (1) and ($1^\prime$) of Proposition \ref{quasi-state basic} and we omit the proof.

\section{Non-degeneracy of spectral norms}\label{non-deg section}
We prove the following lemmas for a non-trivial idempotent $\mathtt{a}$ of $QH_\ast(L)$ with $c^L(\mathtt{a},0)\leq0$.

\begin{lem}\label{disjoint muryoku 1}
Assume that an idempotent $\mathtt{a}$ of $QH_\ast(L)$ satisfies $c^L(\mathtt{a},0)\leq 0$.
For any Hamiltonian function $F,H\colon S^1\times M \to\mathbb{R}$ with $\mathrm{Supp}(F_t)\cap L=\emptyset$ for any $t$,
\[
c^L(\mathtt{a},F\natural H)
=c^L(\mathtt{a},H\natural F)
=c^L(\mathtt{a},H).
\]
\end{lem}
\begin{proof}
Since $\mathrm{Supp}(F)$ is disjoint from $L$, by Proposition \ref{Ostrover}, $c^L(\mathtt{a},F)\leq 2c^L(\mathtt{a},0)\leq0$ and $c^L(\mathtt{a},\bar{F})\leq 2c^L(\mathtt{a},0)\leq0$.
By the triangle inequality and the homotopy invariance ((3), (2) of Proposition \ref{LZ}),
\begin{align*}
&c^L(\mathtt{a},H\natural F)\leq c^L(\mathtt{a},H)+c^L(\mathtt{a},F),\\
 &     c^L(\mathtt{a},H\natural F) +c^L(\mathtt{a},\bar{F})     \geq c^L(\mathtt{a},H).
\end{align*}
Thus
\[c^L(\mathtt{a},H)-c^L(\mathtt{a},\bar{F}) \leq c^L(\mathtt{a},H\natural F)\leq c^L(\mathtt{a},H)+c^L(\mathtt{a},F)\]
Since  $c^L(\mathtt{a},F)\leq 0$ and $c^L(\mathtt{a},\bar{F})\leq 0$, $c^L(\mathtt{a},F)=c^L(\mathtt{a},\bar{F})=0$ and $c^L(\mathtt{a},H\natural F)=c^L(\mathtt{a},H)$.
We can prove $c^L(\mathtt{a},F\natural H)=c^L(\mathtt{a},H)$ similarly.
\end{proof}

\begin{lem}\label{disjoint muryoku 2}
Assume that an idempotent $\mathtt{a}$ of $QH_\ast(L)$ satisfies $c^L(\mathtt{a},0)\leq 0$.
For any $\psi\in\widetilde{\mathrm{Ham}}(M,\omega)$ and any Hamiltonian function $F\colon S^1\times M \to\mathbb{R}$ with $\mathrm{Supp}(F_t)\cap L=\emptyset$ for any $t$,
\[
q_{\mathtt{a},\tilde{\phi}_F}(\tilde{\psi})
=q_{\mathtt{a}}(\tilde{\psi}).
\]
\end{lem}
\begin{proof}

Take a normalized Hamiltonian $H$ generating $\tilde{\psi}$. 
By (2) of Proposition \ref{LZ},
\begin{align*}
&c^L(\mathtt{a},\tilde{\phi}_F^{-k}\tilde{\psi}\tilde{\phi}_F^k)= c^L(\mathtt{a},\bar{F}^{\natural k}\natural H\natural F^{\natural k}),\\
 &     c^L(\mathtt{a},\tilde{\phi}_F^{-k}\tilde{\psi}^{-1}\tilde{\phi}_F^k)= c^L(\mathtt{a},\bar{F}^{\natural k}\natural\bar{H}\natural F^{\natural k}).
\end{align*}
By Lemma \ref{disjoint muryoku 1},
$c^L(\mathtt{a},\bar{F}^{\natural k}\natural H\natural F^{\natural k})=c^L(\mathtt{a},H)=c^L(\mathtt{a},\tilde{\psi})$ and $c^L(\mathtt{a},\bar{F}^{\natural k}\natural \bar{H}\natural F^{\natural k})=c^L(\mathtt{a},\bar{H})=c^L(\mathtt{a},\tilde{\psi}^{-1})$.
Thus, by the definition of $q_{\mathtt{a},\tilde{\phi}_F}(\tilde{\phi})$, we complete the proof.
\end{proof}


\begin{prop}\label{nonzero spec 2}
Assume that an idempotent $\mathtt{a}$ of $QH_\ast(L)$ satisfies $c^L(\mathtt{a},0)\leq 0$.
Let $\hat{U}$ be an open subset of $M$ with $\hat{U}\cap L\neq\emptyset$.
Then there exists an autonomous Hamiltonian function $F\colon M\to\mathbb{R}$ such that $\mathrm{Supp}(F)\subset \hat{U}$ and $c^L(\mathtt{a},F)>0$.
\end{prop}

\begin{proof}
Fix a Riemannian metric on $M$.
For a subset $X$ and a positive number $r$, let $N_r(X)$, $\bar{N}_r(X)$ denote the $r$-neighborhood of $X$, its topological closure, respectively.

Since $\hat{U}\cap L\neq\emptyset$, we can take a point $x$ in $\hat{U}\cap L$.
Then, there exists a positive number $\epsilon$ such that $\bar{N}_{4\epsilon}(\{x\})$ is displaceable from $L$
and $\bar{N}_{4\epsilon}(\{x\})\subset\hat{U}$.

Set $L^\prime=L\setminus\bar{N}_{3\epsilon}(\{x\})$.
Since $\mathrm{dim}L^\prime=\mathrm{dim}L=\frac{1}{2}\mathrm{dim}M$ and $L^\prime$ is not compact, $L^\prime$ is displaceable from $L$.
Thus there is a positive number $\epsilon^\prime$ such that 
$\bar{N}_{2\epsilon^\prime}(L^\prime)$ is also displaceable from $L$.

Set $\epsilon^{\prime\prime}=\min\{\epsilon,\epsilon^\prime\}$.
We define three subsets $U_P,U_S,U_K$ of $M$ by
$U_P=N_{4\epsilon}(\{x\})$,
$U_S=N_{2\epsilon^{\prime\prime}}(L^{\prime})$,
$U_K=M\setminus N_{\epsilon^{\prime\prime}}(L)$.
Then,
\begin{itemize}
\item $U_S$ is displaceable from $L$,
\item  $U_P\cup U_S\cup U_K=M$,
\item $U_K\cap L=\emptyset$.
\end{itemize}

Take a partitions $\{P,S,K\}$ of unity  subordinated to the open cover
$\{U_P,U_S,U_K\}$ \textit{i.e.}
$P,S,K\colon M\to[0,1]$ are smooth functions,
$\mathrm{Supp}(P)\subset U_P$,
$\mathrm{Supp}(S)\subset U_S$,
$\mathrm{Supp}(K)\subset U_K$,
and $P+S+K\equiv1$.

Since $U_S$ is displaceable from $L$ and $\{S,1-S\}=0$, by ($1^\prime$), (2)  of Proposition \ref{quasi-state basic}, 
\[1=\zeta_{\mathtt{a}}^L(1)=\zeta_{\mathtt{a}}^L(1-S).\]

Since $U_K\cap L=\emptyset$, by Lemma \ref{disjoint muryoku 2},
\[
q_{\mathtt{a},\tilde{\phi}_{K}}(\tilde{\phi}_{P})
=q_{\mathtt{a}}(\tilde{\phi}_{P}).
\]
To prove by contradiction, we assume $c^L(\mathtt{a},tP)\leq0$ for any real number $t$.
Since $P$ is a time-independent Hamiltonian function, $\overline{(tP)}=-tP$ for any $t$ and hence $(\tilde{\phi}_{tP})^{-1}=\tilde{\phi}_{-tP}$.
Thus, for any $t$,
\begin{align*}
q_{\mathtt{a}}(\tilde{\phi}_{tP})&=c^L(\mathtt{a},\tilde{\phi}_{tP})+c^L(\mathtt{a},(\tilde{\phi}_{tP})^{-1})\\
     &=c^L(\mathtt{a},\tilde{\phi}_{tP})+c^L(\mathtt{a},\tilde{\phi}_{-tP})\\\
      &=c^L(\mathtt{a},tP)+c^L(\mathtt{a},-tP)\leq0.
\end{align*}

By Proposition \ref{iroirodasu lemma}, $q_{\mathtt{a}}(\tilde{\phi}_{tP})=0$ and therefore $q_{\mathtt{a},\tilde{\phi}_{K}}(\tilde{\phi}_{tP})=q_{\mathtt{a}}(\tilde{\phi}_{tP})=0$.
Hence, by Proposition \ref{pb ineq},
\[\zeta_{\mathtt{a}}^L(P+K)=\zeta_{\mathtt{a}}^L(P)+\zeta_{\mathtt{a}}^L(K).\]

Since $U_P$ and $U_K$ are displaceable from $L$, by Propositions \ref{Ostrover} and \ref{vanishing property},
\[\zeta_{\mathtt{a}}^L(P)=\zeta_{\mathtt{a}}^L(K)=0,\]
and hence $\zeta_{\mathtt{a}}^L(P+K)=0$.
Since $P+S+K\equiv1$, this contradict $1=\zeta_{\mathtt{a}}^L(1-S)$.
Hence we prove that 
$c^L(\mathtt{a},tP)>0$ for some $t$.
Since $\mathrm{Supp}(P)\subset U_P\subset \hat{U}$, we complete the proof.
\end{proof}

\begin{proof}[Proof of Theorem \ref{non-deg of spec norm}]
If $c^L(\mathtt{a},0)>0$, then, by (3) of Proposition \ref{LZ}, $c^L(\mathtt{a},H)+c^L(\mathtt{a},\bar{H})\geq c^L(\mathtt{a},0)> 0$.
Thus we may assume $c^L(\mathtt{a},0)\leq 0$.

Since $\tilde{\phi}_H(L)\neq L$, there is an open subset $U$ of $M$ such that $\tilde{\phi}_H(U)\cap L=\emptyset$ and $U\cap L\neq\emptyset$.
Then, by Lemma \ref{nonzero spec 2},  there exists a Hamiltonian function $F\colon S^1 \times U\to\mathbb{R}$ with compact support on $S^1\times U$ such that $c^L(\mathtt{a},F)>0$.
Thus, by Proposition \ref{Ostrover}, $c^L(\mathtt{a},H)+c^L(\mathtt{a},\bar{H})\geq c^L(\mathtt{a},F)> 0$.
\end{proof}

\section{The energy capacity inequality and heavy, superheavy subsets on  a torus}\label{Proof of the energy capacity inequality}



\begin{prop}\label{spectral detemined}
Let $L$ be a Lagrangian submanifold of a closed symplectic manifold $(M,\omega)$ with $\omega(\pi_2(M,L))=0$, $N_L\geq2$ and $QH_\ast(L)\neq0$
(Note that $\omega(\pi_2(M,L))=0$ implies monotonicity of $L$).
Assume that a Hamiltonian function $H\colon M\to\mathbb{R}$ with compact support is 
$L$-slow.
Then
\[c^L([L],H)=m(H).\]
\end{prop}

\begin{proof}
Since $\omega(\pi_2(M,L))=0$ and $H$ is 
$L$-slow, $\mathrm{Spec}(H;L)=\{m(H),0\}$, $\mathrm{Spec}(\bar{H};L)=\{-m(H),0\}$ and $\phi_H(L)\neq L$.
By Theorem \ref{non-deg of spec norm} and $\phi_H(L)\neq L$, $c^L([L],H)+c^L([L],\bar{H})>0$.
Thus, by the spectrality ((6) of Proposition \ref{LZ}), $c^L([L],H)=m(H)$ and $c^L([L],\bar{H})=0$.
\end{proof}

\begin{proof}[Proof of Theorem \ref{energy capacity ineq}]
Let $F$ be an $L\cap U$-simple and $L\cap U$-slow Hamiltonian function with compact support on $U$.
Define the Hamiltonian function $\hat{F}\colon M\to\mathbb{R}$ by
\begin{equation*}
\hat{F}(x)=
\begin{cases}
F(x) & \text{if $x\in U$},\\
0 & \text{if $x\notin U$}.
\end{cases}
\end{equation*}
Then, by the definition of $L$-simplicity, $\hat{F}$ is an $L$-slow Hamiltonian function with compact support on $M$.
Then, by Proposition \ref{spectral detemined}, $c^L([L],\hat{F})=m(F)$.
By Corollary \ref{Hofer Ostrover}, $c^L([L],\hat{F})\leq \bar{E}(U;L)$.
Thus $m(F)\leq  \bar{E}(U;L)$.
By taking the supremum over all $L\cap U$-simple and $L\cap U$-slow functions with compact support on $U$, we complete the proof.
\end{proof}

To prove (3) of Proposition \ref{torus summary},
we use the following lemma.
\begin{lem}\label{torus ham}
Let $f\colon \mathbb{R}/R\mathbb{Z}\to\mathbb{R}$ be a smooth function with only two critical value $v_{\max}>v_{\min}$ and define a Hamiltonian function $F\colon T_R^2\to\mathbb{R}$ on $T_R^2$ by $F(x,y)=f(y)$.
Then $c^L([L_{\mathrm{m}}^s],F)=v_{\max}$.
\end{lem}
\begin{proof}
By the definition of the spectrum, $\mathrm{Spec}(F;L_{\mathrm{m}}^s)=\{v_{\max},v_{\min}\}$ and $\mathrm{Spec}(\bar{F};L_{\mathrm{m}}^s)=\{-v_{\max},-v_{\min}\}$.
However, by Theorem \ref{non-deg of spec norm}, $c^{L_{\mathrm{m}}^s}([L_{\mathrm{m}}^s],F)+c^{L_{\mathrm{m}}^s}([L_{\mathrm{m}}^s],\bar{F})>0$.
Thus by the spectrality ((6) of Proposition \ref{LZ}), $c^{L_{\mathrm{m}}^s}([L_{\mathrm{m}}^s],F)=v_{\max}$ and $c^{L_{\mathrm{m}}^s}([L_{\mathrm{m}}^s],\bar{F})=-v_{\min}$.
\end{proof}

\begin{proof}[Proof of (3) of Proposition \ref{torus summary}]
Fix a Hamiltonian function $H\colon S^1\times T_R^2\to\mathbb{R}$.
Take a function $f\colon \mathbb{R}/R\mathbb{Z}\to\mathbb{R}$ with only two critical values $v_{\max}=\inf_{S^1\times L_{\mathrm{l}}^t}H>v_{\min}$ such that $F(x,y)\leq H(x,y)$ where $F\colon T_R^2\to\mathbb{R}$ is a Hamiltonian function on $T_R^2$ defined by $F(x,y)=f(x)$.
Then by Lemma \ref{torus ham}, $c^{L_{\mathrm{m}}^s}([L_{\mathrm{m}}^s],kF)=kv_{\max}=k\inf_{S^1\times L_{\mathrm{l}}^t}H$ for any positive integer $k$.
By the monotonicity of Lagrangian spectral invariants, $\zeta_{[L_{\mathrm{m}}^s]}^{L_{\mathrm{m}}^s}(H)\geq\zeta_{[L_{\mathrm{m}}^s]}^{L_{\mathrm{m}}^s}(F)=\inf_{S^1\times L_{\mathrm{l}}^t}H$.
\end{proof}

We can prove the following proposition similarly to (3) of Proposition \ref{torus summary}.

\begin{prop}\label{Riemannian}
Let $L$ be a monotone Lagrangian submanifold of a closed symplectic manifold $(M,\omega)$ with $N_L\geq2$, $\omega(\pi_2(M,L))=\{0\}$ and $QH_\ast(L)\neq0$.
Let $L^\prime$ be a compact Lagrangian submanifold of $(M,\omega)$.
Assume that there are a  Weinstein coordinate $w\colon W\to T^\ast L^\prime$ of $L^\prime$ (note that $w(L^\prime)$ is the zero section of $T^\ast L^\prime$) and a point $q$ of $L^\prime$ such that
\[w(W\cap L)=T_q^\ast L^\prime\cap \mathrm{Im}(w).\] 
Also assume that $L^\prime$ is diffeomorphic to a torus or a negatively curved Riemannian manifold and the map $\iota_\ast\pi_1(L^\prime,q)\to\pi_1(M,q)\to\pi_1(M,L,q)$ is injective,
where $\iota$ is the homomorphism induced from the inclusion map $\iota\colon (L^\prime,q,q)\to(M,L,q)$.
Then $L^\prime$ is $(L,\mathtt{a})$-heavy for any non-trivial idempotent $\mathtt{a}$.
\end{prop}

For example, we can apply Proposition \ref{Riemannian} to a Riemann surface with higher genus.

We give proofs of the other parts of Proposition \ref{torus summary}.

\begin{proof}[Proof of (1) of Proposition \ref{torus summary}]
If $s=s^\prime$, $L_{\mathrm{m}}^{s^\prime}=L_{\mathrm{m}}^s$ is an $L_{\mathrm{m}}^s$-stem and in particular, $(L_{\mathrm{m}}^s,[L_{\mathrm{m}}^s])$-superheavy.
\end{proof}

\begin{proof}[Proof of (2) of Proposition \ref{torus summary}]
If $s\neq s^\prime$, $L_{\mathrm{m}}^{s^\prime}\cap L_{\mathrm{m}}^s=\emptyset$.
Thus, by (1) of Proposition \ref{torus summary} and Proposition \ref{hv shv intersect}, $L_{\mathrm{m}}^{s^\prime}$ is not $(L_{\mathrm{m}}^s,[L_{\mathrm{m}}^s])$-heavy.
\end{proof}

\begin{proof}[Proof of (4) of Proposition \ref{torus summary}]
Take $t^\prime\in\mathbb{R}/\mathbb{Z}$ with $t^\prime\neq t$.
Then, by (3) of Proposition \ref{torus summary}, $L_{\mathrm{l}}^{t^\prime}$ is $(L_{\mathrm{m}}^s,[L_{\mathrm{m}}^s])$-heavy.
Since $L_{\mathrm{l}}^t\cap L_{\mathrm{l}}^{t^\prime}=\emptyset$, by Proposition \ref{hv shv intersect}, $L_{\mathrm{l}}^t$ is not $(L_{\mathrm{m}}^s,[L_{\mathrm{m}}^s])$-superheavy.
\end{proof}

\begin{proof}[Proof of (5) of Proposition \ref{torus summary}]
For any $s,s^\prime,t$, $L_{\mathrm{m}}^{s^\prime}\cup L_{\mathrm{l}}^t$ is an $L_{\mathrm{m}}^s$-stem and in particular, $(L_{\mathrm{m}}^s,[L_{\mathrm{m}}^s])$-superheavy.
\end{proof}

\section{Proof of Theorem \ref{main monotone just}}\label{prf of main theorem}
In order to prove Theorem \ref{main monotone just}, we give an upper bound of the spectral invariant associated with a Hamiltonian function $F\colon S^1\times M\times\mathbb{A}(R)\to\mathbb{R}$ such that  $\mathsf{Ch}(F;L\times \mathbb{A}(R)_0,L\times \mathbb{A}(R)_s,\alpha_s)=\emptyset$.
Here, for $R=(R_1,\ldots,R_n)\in (\mathbb{R}_{>0})^n$ and a positive real number $\epsilon$ with  $\epsilon <\min\{R_1,\ldots,R_n\}$, let $R(\epsilon)$ denote $(R_1-\epsilon,\ldots,R_n-\epsilon)$ and 
$I_{R(\epsilon)}^n$ denote $(\epsilon,R_1)\times\cdots\times (\epsilon,R_n)\subset (\mathbb{R}_{>0})^n$.
For $R=(R_1,\ldots,R_n)\in (\mathbb{R}_{>0})^n$, $T_R^n$ is defined to be $\mathbb{R}/R_1\mathbb{Z}\times\cdots\mathbb{R}/R_n\mathbb{Z}$ and we set the symplectic form $\omega_0=dp_1\wedge dq_1+\cdots+dp_n\wedge dq_n$ on $T^n_R\times T^n$ with the coordinates $(p,q)=(p_1,\ldots,p_n,q_1,\ldots,q_n)$.

\begin{prop}\label{Irie and Seyfaddini}
Let $L$ be a $\lambda$-monotone Lagrangian submanifold of a $2m$-dimensional closed symplectic manifold $(M,\omega)$ with $N_L\geq 2$ and $QH_\ast(L)\neq0$.
Let $s=(s_1,\ldots,s_n)$ and $R=(R_1,\ldots,R_n)$ be elements of $\mathbb{Z}^n$ and $(\mathbb{R}_{>0})^n$, respectively.
For a positive real number $\epsilon$ with $3\epsilon <\min\{R_1,\ldots,R_n\}$, let $U_\epsilon$ be the open subset of $T^n_R\times T^n$ defined by
\[U_\epsilon=\{(p,q)\in T^n_R\times T^n; p\in I_{R(3\epsilon)}^n\}.\]
We fix the symplectic form $\operatorname{pr}_1^\ast\omega +\operatorname{pr}_2^\ast\omega_0$ on $ M\times T_R^n\times T^n$, where $\operatorname{pr}_1\colon  M\times T_R^n\times T^n\to M$ and $\operatorname{pr}_2\colon  M\times T_R^n\times T^n \to T_R^n\times T^n$ are the projections defined by $\operatorname{pr}_1(x,p,q)=x$ and $\operatorname{pr}_2(x,p,q)=(p,q)$.
Then for any Hamiltonian function $F\colon S^1\times M\times U_\epsilon\to\mathbb{R}$ with compact support such that $\mathsf{Ch}(F;L\times (T^n_R\times \{0\}),L\times (T^n_R\times \{s\}),(0_L,\alpha_s))=\emptyset$,
\[c^{L\times (T^n_R\times \{0\})}([L\times (T^n_R\times \{0\})],F)<2\sum_{i=1}^n R_i\cdot |s_i|+\lambda(m+n).\]

\end{prop}
To prove Proposition \ref{Irie and Seyfaddini}, we use the following proposition.
For a smooth path $z\colon [0,1]\to M$, let $\mathrm{ev}(z)$ denote the point $z(0)$.



\begin{prop}\label{cL}
Let $W$ be an open subset of a $2w$-dimensional connected closed symplectic manifold $(\hat{W},\omega)$, $\hat{Z}_0,\hat{Z}_1$ compact $\lambda$-monotone Lagrangian submanifolds of $\hat{W}$  with $N_{\hat{Z}_0}\geq 2$ and $QH_\ast(\hat{Z}_0)\neq0$ and $\alpha$ a homotopy class of $\pi_1(\hat{W},\hat{Z}_0\cup\hat{Z}_1)$.
Assume that a Hamiltonian function $H\colon \hat{W}\to\mathbb{R}$ satisfies the following conditions.
\begin{itemize}
\item $\mathrm{Supp}(H)\subset W$
\item $\phi_H(\hat{Z}_1)=\hat{Z}_0$ and $[\gamma_H^x]=\bar{\alpha}$ for any point $x$ in $Z_1$ where $Z_1=W\cap\hat{Z}_1$,
\item 
Set $Z_0=W\cap \hat{Z}_0$.
On some Weinstein neighborhood of $Z_0$, $H$ is the pullback of a Morse function $\tilde{H}\colon Z_0\to\mathbb{R}$, 
\item $\mathrm{ev}(\mathsf{Ch}(H;Z_0))=\operatorname{Crit}(H|_{Z_0})$,
\item $\operatorname{ind}_{Morse}(x)=\operatorname{ind}_{\mathrm{CZ}}([x,c_x])$ for any point $x$ in $\mathrm{Crit}(H|_{Z_0})$.
\end{itemize}

Then for any Hamiltonian function $F\colon S^1\times W\to\mathbb{R}$ with compact support such that $\mathsf{Ch}(F;Z_0,Z_1,\alpha)=\emptyset$,
 \[c^{\hat{Z}_0}([\hat{Z}_0],F)\leq 2||H||_{L^\infty}+\lambda w.\]

\end{prop}

\begin{proof}
To give an upper bound of the spectral invariant associated with $F$, we consider the concatenation of $\phi_F^1$ and a Hamiltonian diffeomorphism $\phi_H^1$ with trajectories in $\bar{\alpha}$.

Let  $K \colon S^1\times \hat{W}\to\mathbb{R}$ be a Hamiltonian function defined by
\begin{equation*}
K(t,x)=
\begin{cases}
\frac{\partial\chi}{\partial t}(t)H(\chi(t),x) & \text{when }t\in[0,\frac{1}{2}], \\
\frac{\partial\chi}{\partial t}(t-\frac{1}{2})F(\chi(t-\frac{1}{2}),x)& \text{when }t\in[\frac{1}{2},1],
\end{cases}
\end{equation*}
where $\chi\colon[0,\frac{1}{2}]\to[0,1]$ is the function defined in the proof of Proposition \ref{Ostrover}.
We claim
 \[c^{\hat{Z}_0}([\hat{Z}_0],K)\leq ||H||_{L^\infty}+\lambda w.\]

Let $[z,\hat{z}]\in\tilde{\Omega}_0(\hat{Z}_0)$ and define $x$ by $x=z(0)\in\hat{Z}_0$.
If $x\in Z_0$, by the assumption of $H$, $[\gamma_H^x]=\bar{\alpha}$.
Since the path $\gamma_{K}^x$ is the concatenation of the paths $\gamma_H^x$ and $\gamma_F^{\phi_H(x)}$ up to parameter change, $\mathsf{Ch}(F;Z_0,Z_1,\alpha)=\emptyset$ implies 
$\gamma_K^x\notin \Omega_0(\hat{Z}_0)$ for any $x\in Z_0$.
If $x\notin Z_0$, then $\phi_H(x)\notin Z_0$.
Thus $\gamma_K^x$ is equal to $\gamma_H^x$ up to parameter change and $\int_0^1H(t,\gamma_H^x(t))dt=\int_0^1K(t,\gamma_K^x(t))dt$.
Therefore there exists the natural inclusion map $\iota\colon  \tilde{\mathsf{Ch}}(K;Z_0)\to \tilde{\mathsf{Ch}}(H;Z_0)$ which preserves values of the action functionals and the Conley-Zehnder indices.

We give an upper bound of the critical value of the action functional $\mathcal{A}_K^{\hat{Z}_0}$ which attains the fundamental class $[\hat{Z}_0]$ of $QH_\ast(\hat{Z}_0)$.
Since every element of $\mathsf{Ch}(H;\hat{Z}_0)$ is a constant path, every element of $\mathsf{Ch}(K;\hat{Z}_0)$ is also a constant path.
Since  $\mathsf{Ch}(K;\hat{Z}_0)$ is a finite set and $\hat{Z}_0$ is monotone, $\mathcal{A}_K^{\hat{Z}_0}(\tilde{\mathsf{Ch}}(H;\hat{Z}_0))$ is a discrete subset of $\mathbb{R}$.
Thus $c^{\hat{Z}_0}([\hat{Z}_0],K)$ is attained by a 1-length trajectory of the Conley-Zehnder index $w$ that is the dimension of the fundamental class.
Since every element of $\mathsf{Ch}(K;\hat{Z}_0)$ is a constant path, there exist a point $x$ in $\mathrm{Crit}(\tilde{H})$ and $A\in\pi_2(\hat{W},\hat{Z}_0)$ such that $\operatorname{ind}_{\mathrm{CZ}}([x,c_x\natural A])=w$ and $c^{\hat{Z}_0}([\hat{Z}_0],K)=\mathcal{A}_K^{\hat{Z}_0}([x,c_x\natural A])$.
Then, by the assumption,
\begin{align*}
&\operatorname{ind}_{\mathrm{Morse}}(x)-\mu(A)\\
      &=\operatorname{ind}_{\mathrm{CZ}}([x,c_x])-\mu(A)\\
      & =\operatorname{ind}_{\mathrm{CZ}}([x,c_x\natural A])\\
      & =w.
\end{align*}
Since $0\leq\operatorname{ind}_{\mathrm{Morse}}(x)\leq w$,
\[-w\leq \mu(A)\leq 0.\]
Since $\iota$ preserves values of the action functionals,
\begin{align*}
\mathcal{A}_K^{\hat{Z}_0}([x,c_x\natural A])&=\mathcal{A}_H^{\hat{Z}_0}([x,c_x\natural A])\\
& =H(x)-\omega(A)\\
      & =H(x)-\lambda \mu(A).
\end{align*}
Thus, by $-w\leq \mu(A)\leq 0$ and $\lambda\geq0$, $c^{\hat{Z}_0}([\hat{Z}_0],K)\leq||H||_{L^\infty}+\lambda w$.
By $||\bar{H}||_{L^\infty}=||H||_{L^\infty}$, the Lipschitz continuity and the homotopy invariance for spectral invariants (Proposition \ref{LZ} (1) and (2)) imply
\begin{align*}
c^{\hat{Z}_0}([\hat{Z}_0],F)& \leq c^{\hat{Z}_0}([\hat{Z}_0],K)+||\bar{H}||_{L^\infty}\\
            & =2||H||_{L^\infty}+\lambda w.
\end{align*}
\end{proof}
The idea of using a Hamiltonian function $H$ satisfying the above conditions comes from Irie's paper \cite{Ir}.
Seyfaddini's techniques of using the monotonicity assumption \cite{S} are also important in our proof.

To prove Proposition \ref{Irie and Seyfaddini}, we construct the Hamiltonian function $H$ in Proposition  \ref{cL} using $H^{R,\epsilon,e}$ given by the following lemma.
\begin{lem}\label{function on torus}
Let $R$, $\epsilon$ be positive real numbers such that $3\epsilon <R$. 
Let $w_1$ and $w_2$ denote the points $(\epsilon,0)$ and $(2\epsilon,0)$ in $T^1_R\times T^1$, respectively.
For an integer $e$, there exists a Hamiltonian function $H^{R,\epsilon,e}\colon T^1_R\times T^1\to\mathbb{R}$ satisfying the following conditions.

\begin{itemize}
\item  $H^{R,\epsilon,e}(p,q)=-ep$ on $U_\epsilon=(3\epsilon,R)\times T^1$,
\item $\operatorname{Crit} (H^{R,\epsilon,e}|_{T_R^1\times  \{0\}})=\{w_1,w_2\}$,
\item $H^{R,\epsilon,e}|_{T_R^1\times  \{0\}}$ is a Morse function,
\item $||H^{R,\epsilon,e}||_{L^\infty}<(R-\epsilon)\cdot |e|$,
\item $dH^{R,\epsilon,e}$ is $C^1$-small near $w_1,w_2$,
\item $\operatorname{ev}(\mathsf{Ch}(H^{R,\epsilon,e};T_R^1\times  \{0\}))=\operatorname{Crit}(H^{R,\epsilon,e}|_{T_R^1\times  \{0\}})$.
\end{itemize}
Here $\operatorname{Crit}(H^{R,\epsilon,e})$ is the set of critical points of $H^{R,\epsilon,e}$.
\end{lem}
\begin{proof}

Let $\underline{H}^{R,\epsilon,e}\colon T^1_R\to\mathbb{R}$ be a function satisfying the following conditions.

\begin{itemize}
\item  $\underline{H}^{R,\epsilon,e}(p)=-ep$ on $U_\epsilon=(3\epsilon,R)$,
\item $\operatorname{Crit} (\underline{H}^{R,\epsilon,e})=\{\epsilon,2\epsilon\}$,
\item $\underline{H}^{R,\epsilon,e}$ is a Morse function,
\item $||\underline{H}^{R,\epsilon,e}||_{L^\infty}<(R-\epsilon)\cdot |e|$,
\item  $d\underline{H}^{R,\epsilon,e}$ is $C^1$-small near $p=\epsilon,2\epsilon$.
\end{itemize}
Define the Hamiltonian function $H^{R,\epsilon,e}\colon T^1_R\times T^1\to\mathbb{R}$ by $H^{R,\epsilon,e}(p,q)=\underline{H}^{R,\epsilon,e}(p)$.
Then, the last condition follows from $C^1$-smallness of $d\underline{H}^{R,\epsilon,e}$ near $p=\epsilon,2\epsilon$.
The other conditions immediately follow from the conditions of $\underline{H}^{R,\epsilon,e}\colon T^1_R\to\mathbb{R}$.
\end{proof}

\begin{proof}[Proof of Proposition \ref{Irie and Seyfaddini}]

To use Proposition \ref{cL}, we construct the Hamiltonian function $H$.
Define the Hamiltonian function $H^\prime\colon T^n_R\times T^n\to\mathbb{R}$ by
\[H^\prime(p,q)=\sum_{i=1}^nH^{R_i,\epsilon_i,e_i}(p_i,q_i).\]
Then $[\gamma_{H^\prime}^x]=-e \in\pi_1(T_R^n\times T^n,T_R^n\times\{0\})$ for any $x\in U_\epsilon$.
Thus we can take a neighborhood $V$ of $\bar{U}_\epsilon$ such that
\[\operatorname{ev}(\mathsf{Ch}(H^\prime;L\times (T_R^n\times\{0\})))\cap \bar{V}=\emptyset.\]
In order to compute the spectral invariant associated with $F$, we take a perturbation of $\mathrm{pr}_1^\ast H^\prime$
(recall that $\mathrm{pr}_1\colon M\times T_R^n\times T^n\to M$ is the first projection).
Let $\rho\colon T^n_R\times T^n\to[0,1]$ be a function such that
\begin{equation*}
\rho(p,q)=
\begin{cases}
1 & \text{for any }(p,q)\in (T^n_R\times T^n)\setminus V, \\
0 & \text{for any }(p,q)\in U_\epsilon.
\end{cases}
\end{equation*}
Let $G\colon M\to\mathbb{R}$ be a function satisfying the following conditions.
\begin{itemize}
\item 
On some Weinstein neighborhood of $L$, $G$ is the pullback of a $C^2$-small Morse function $\tilde{G}\colon L\to\mathbb{R}$, 
\end{itemize}
Define the Hamiltonian function $H\colon M\times T^n_R\times T^n\to\mathbb{R}$ by
\[H(y,p,q)=H^\prime(p,q)+\rho(p,q)\cdot G(y).\]
Since $\tilde{G}$ is sufficiently $C^2$-small, then 
\begin{itemize}
\item $\operatorname{ev}(\mathsf{Ch}(H;L\times (T_R^n\times\{0\})))\cap(M\times V)=\emptyset,$ and
\item there exist only finitely many points $y_1, \ldots, y_k$ in $M$ such that $\operatorname{Crit}(G)=\operatorname{ev}(\mathsf{Ch}(tG;L))=\{y_1, \ldots, y_k\}$ for any $t\in(0,1]$.
\end{itemize}
Since  $\operatorname{ev}(\mathsf{Ch}(H;L\times (T_R^n\times\{0\})))\cap(M\times V)=\emptyset$,
$\operatorname{ev}(\mathsf{Ch}(H;L\times (T_R^n\times\{0\})))\subset M\times((T_R^n\times T^n)\setminus V)$.
Since $\rho(p,q)=1$ for any $(p,q)\in (T^n_R\times T^n)\setminus V$,
\[H(y,p,q)=H^\prime(p,q)+G(y),\]
for any $(y,p,q)\in M\times((T_R^n\times T^n)\setminus V)$.
  Thus
  \[\operatorname{ev}(\mathsf{Ch}(H;L\times (T_R^n\times\{0\})))=\{(y_i,(w_{j_1},\ldots,w_{j_n}),(0,\ldots,0))\}_{i\in\{1,\ldots,k\}, j_1,\ldots,j_n\in\{1,2\}}=\operatorname{Crit}(H|_{L\times (T_R^n\times\{0\})}).\]
By conditions of $H^{R_i,\epsilon,e_i}$'s and $G$,
$dH$ is $C^1$-small near critical points of $H|_{L\times (T_R^n\times\{0\})}$.
Thus
  \[\operatorname{ind}_{\mathrm{Morse}}(x)=\operatorname{ind}_{\mathrm{CZ}}([x,c_x]),\]
  for any point $x$ in $\operatorname{Crit}(H|_{L\times (T_R^n\times\{0\})})$.  
Hence $H$ satisfies the conditions of Proposition \ref{cL} and thus we apply Proposition \ref{cL}.

By Proposition \ref{cL} and $||\bar{H}||_{L^\infty}=||H||_{L^\infty}$, the Lipschitz continuity and the homotopy invariance for spectral invariants (Proposition \ref{LZ} (1) and (2)) imply
\begin{align*}
c^{L\times \mathbb{A}(R)_0}([L\times \mathbb{A}(R)_0],F)& \leq 
2||H||_{L^\infty}+\lambda(m+n)\\
            & <2(\sum_{i=1}^n(R_i-\epsilon)\cdot |e_i|+||G||_{L^\infty})+\lambda(m+n).
\end{align*}
Since the Morse function $G$ is sufficiently $C^2$-small,
\[c^{L\times \mathbb{A}(R)_0}([L\times \mathbb{A}(R)_0],F)<2\sum_{i=1}^nR_i\cdot |e_i|+\lambda(m+n).\]
\end{proof}

\begin{proof}[Proof of Theorem \ref{main monotone just}]
Fix a positive real number $\epsilon$ such that $\epsilon <\min\{R_1,\ldots,R_n\}$ and take a Hamiltonian function $F\colon S^1\times M\times I^n_{R(\epsilon)}\times T^n\to\mathbb{R}$ with compact support such that $F|_{S^1\times Z\times T^n}\geq 2\sum_{i=1}^nR_i\cdot |e_i|+\max\{0,\lambda(m+n)\}$.
Assume $\mathsf{Ch}(F;L\times (T_R^n\times\{0\}),L\times (T_R^n\times\{s\}),(0_L,\alpha_s))= \emptyset$.
Then, Proposition \ref{Irie and Seyfaddini} and the triangle inequality ((3) of Proosition \ref{LZ}) imply
\[\zeta_{[L\times \mathbb{A}(R(\epsilon))_0]}^{L\times \mathbb{A}(R(\epsilon))_0}(F)< 2\sum_{i=1}^nR_i\cdot |e_i|+\max\{0,\lambda(m+n)\}.\]

By (3) of Proposition \ref{torus summary} and Theorem \ref{product}, $Z\times T^n$ is a $L\times \mathbb{A}(R(\epsilon))_0$-heavy subset.
Since Corollary \ref{module find 2} implies that $Z\times T^n$ is $[L\times \mathbb{A}(R(\epsilon))_0]$-heavy,  by Definition \ref{definition of heavy},
\[\zeta_{[L\times \mathbb{A}(R(\epsilon))_0]}^{L\times \mathbb{A}(R(\epsilon))_0}(F)\geq 2\sum_{i=1}^nR_i\cdot |e_i|+\max\{0,\lambda(m+n)\}.\]

These two inequalities contradict.
Since any Hamiltonian function $F\colon S^1\times M\times I^n_{R(0)}\times T^n\to\mathbb{R}$ with compact support has support in $S^1\times M \times I^n_{R(\epsilon)}\times T^n$ for some $\epsilon$, we complete the proof of Theorem \ref{main monotone just}.
\end{proof}

\section{Flexibility result}\label{flex section}

To prove Proposition \ref{displaceable main}, we use the following proposition which is a chord-version of Proposition 3.3.2 of \cite{BPS}.

\begin{prop}\label{displaceable prop}
Let $(N,\omega)$ be an open symplectic manifold, $Y_0,Y_1,Z$ subsets of $N$ and $\alpha$ a homotopy class in $\pi_1(N,Y_0\cup Y_1)$.
Assume that $Z$ is compact and there exists a Hamiltonian function $H\colon S^1\times N\to\mathbb{R}$ with compact support satisfying the following conditions.
\begin{itemize}
\item[(1)] $Y_0\cap\phi_H(Z)=\emptyset$.
\item[(2)] $\mathsf{Ch}(H;Y_1,Y_0,\bar{\alpha})=\emptyset$.
\end{itemize}
Then $C_{BEP}(N,Y_0,Y_1,Z,\alpha)=+\infty$.
\end{prop}

\begin{proof}
Fix a positive number $K$.
Since $Z$ is compact and $Y_0\cap\phi_H(Z)=\emptyset$, there exists an open neighborhood $U$ of $Z$ such that $Y_0\cap\phi_H(U)=\emptyset$.
Let $F\colon U\to\mathbb{R}$ be a Hamiltonian function with compact support such that $\inf_ZF+\sup_{S^1\times N}\bar{H}\geq K$.
Since $F\natural \bar{H}(t,x)=F(t,x)+\bar{H}(t,(\phi_F^t)^{-1}(x))$ and $\inf_ZF+\sup_{S^1\times N}\bar{H}\geq K$, $\inf_{S^1\times Z}F\natural \bar{H}\geq K$.

We claim that $\mathsf{Ch}(F\natural \bar{H};Y_0,Y_1,\alpha)=\emptyset$.
On the contrary, we assume that $\mathsf{Ch}(F\natural \bar{H};Y_0,Y_1,\alpha)\neq\emptyset$, take an element $z\colon [0,1]\to N$ of $\mathsf{Ch}(F\natural \bar{H};Y_0,Y_1,\alpha)$ and set $y=z(0)$.
Since $y\in Y_0$, $\phi_{\bar{H}}(y)\in\phi_{\bar{H}}(Y_0)=(\phi_H)^{-1}(Y_0)$.
Since $Y_0\cap\phi_H(U)=\emptyset$, $(\phi_H)^{-1}(Y_0)\cap U=\emptyset$ and thus $\phi_{\bar{H}}(y)\notin U$ .
Hence, since $F$ has support in $U$, $\phi_{F\natural \bar{H}}(y)=\phi_{\bar{H}}(y)$ and $\gamma_{\bar{H}}^y\in\mathsf{Ch}(\bar{H};Y_0,Y_1,\alpha)$.
$\gamma_{\bar{H}}^y\in\mathsf{Ch}(\bar{H};Y_0,Y_1,\alpha)$ implies $\gamma_{H}^{(\phi_H)^{-1}(y)}\in\mathsf{Ch}(H;Y_1,Y_0,\bar{\alpha})$ and it contradicts with $\mathsf{Ch}(H;Y_1,Y_0,\bar{\alpha})=\emptyset$.
Thus we complete the proof of $\mathsf{Ch}(F\natural \bar{H};Y_0,Y_1,\alpha)=\emptyset$.

Since $K$ is any positive number, $C_{BEP}(N,Y_0,Y_1,Z,\alpha)=+\infty$.
\end{proof}

\begin{proof}[Proof of Proposition \ref{displaceable main}]
To use Proposition \ref{displaceable prop}, we construct a Hamiltonian function $\hat{H}\colon S^1\times M\times \mathbb{A}(R)\to\mathbb{R}$ such that $(L\times \mathbb{A}(R)_0)\cap\phi^1_{\hat{H}}(Z\times \mathbb{A}(0))=\emptyset$ and $\mathsf{Ch}(\hat{H};L\times \mathbb{A}(R)_s,L\times \mathbb{A}(R)_0,\bar{\alpha})=\emptyset$.
Fix a sufficiently small positive number $\epsilon$.
By the assumption, we can take a Hamiltonian function $H\colon S^1\times M\to\mathbb{R}$ with compact support such that $||H||<E(Z;L)+\epsilon$ and $L\cap \phi_H^1(Z)=\emptyset$.
Since $|s_k|\cdot R_k>E(Z;L)$ and $\epsilon$ is sufficiently small, we can take a function $\rho_k\colon (-R_k,R_k)\to\mathbb{R}$ with compact support and such that
\begin{itemize}
\item $\rho_k\equiv 1$ in a neighborhood of \{0\},
\item $|\dot{\rho}_k(x)|<|s_k|\cdot (E(Z;L)+\epsilon)^{-1}$ for any $x\in (-R_k,R_k)$.
\end{itemize}
For $i\neq k$, we take a function $\rho_i \in C_c^\infty(-R_i,R_i)$ with $\rho_i\equiv 1$ in a neighborhood of $\{0\}$.
We define the Hamiltonian function $\hat{H}\colon S^1\times M\times \mathbb{A}(R)\to\mathbb{R}$ by
\[\hat{H}(t,x,p,q)=\prod_i\rho_i(p_i)\cdot H(t,x).\]
Then
\[(X_{\hat{H}}^t)_{(x,p,q)}=(\prod_i\rho_i(p_i)\cdot (X_H^t)_x,0,\ldots,0,\dot{\rho}_1(p_1)\cdot H(t,x),\ldots,\dot{\rho}_n(p_n)\cdot H(t,x) ).\]
Since  $\rho_i\equiv 1$ in a neighborhood of $\{0\}$, $(L\times \mathbb{A}(0))\cap\phi^1_{\hat{H}}(Z\times \mathbb{A}(0))=\emptyset$.
By the above comuputation of $X_{\hat{H}}$,  $\phi^1_{\hat{H}}(Z\times \mathbb{A}(0))\subset M\times\mathbb{A}(0)$ and thus $(L\times \mathbb{A}(R)_0)\cap\phi^1_{\hat{H}}(Z\times \mathbb{A}(0))=\emptyset$.
Since $|\dot{\rho}_k|<|s_k|\cdot (E(Z;L)+\epsilon)^{-1}$ and  $\int_0^1||H_t||_{L^\infty}dt=||H||<E(Z;L)+\epsilon$, $\int_0^1|\dot{\rho}_k(p_k)|\cdot |H(t,x)|dt=\dot{\rho}_k(p_k)\cdot\int_0^1|H(t,x)|dt$ is smaller than $|s_k|$ and hence $\mathsf{Ch}(\hat{H};L\times \mathbb{A}(R)_s,L\times \mathbb{A}(R)_0,\bar{\alpha}_s)=\emptyset$.
Thus Proposition \ref{displaceable prop} implies 
\[C(M,L,Z,R,s)=C_{BEP}(M\times \mathbb{A}(R),L\times \mathbb{A}(R)_0,L\times \mathbb{A}(R)_s,Z\times \mathbb{A}(0),\alpha_s)=+\infty.\]
\end{proof}


\end{document}